\documentclass{amsart}
\usepackage{ifpdf}
\usepackage{amsmath, amsthm, amsfonts, amssymb, amscd}
\usepackage[active]{srcltx}
\usepackage{color}
    \ifpdf

      \usepackage[pdftex]{graphicx}  % note the x at the end
      \usepackage[pdftex]{hyperref}

    \else

      \usepackage[dvips]{graphicx}  % note the x at the end
      \newcommand{\href}[2]{#2}

    \fi

\DeclareMathOperator{\cPE}{\rm{c}_{\mc{E}}}

\DeclareMathOperator{\bdPE}{\rm{b}_{\mc{E}}}

\DeclareMathOperator{\pr}{\rm{pr}}

\DeclareMathOperator{\inter}{\rm{int}}
\DeclareMathOperator{\bd}{\partial}

\DeclareMathOperator{\diam}{\rm{diam}}

\DeclareMathOperator{\fix}{\rm{Fix}}

\newcommand{\mc}{\mathcal}

\newcommand{\ol}{\overline}
\newcommand{\R}{\mathbb{R}}\newcommand{\N}{\mathbb{N}}
\newcommand{\Z}{\mathbb{Z}}
\newcommand{\T}{\mathbb{T}}
\newcommand{\D}{\mathbb{D}}\newcommand{\A}{\mathbb{A}}
\renewcommand{\SS}{\mathbb{S}}

\newcommand{\sm}{\setminus}

\newcommand{\id}{\mathrm{Id}}

\newcommand{\ie}{i.e.\ }

\newcommand{\cB}{\mathcal{B}}

\newtheorem{theorem}{Theorem}[section] %[section]
\newtheorem{corollary}[theorem]{Corollary}
\newtheorem{lemma}[theorem]{Lemma}
\newtheorem{proposition}[theorem]{Proposition}
\newtheorem*{proposition*}{Proposition}

\newtheorem{claim}{Claim}
\newtheorem*{question*}{Question}
\newtheorem*{theorem*}{Theorem}
\newtheorem*{claim*}{Claim}

\newtheorem{theoremain}{Theorem}

\newtheorem{corollarymain}[theoremain]{Corollary}

\theoremstyle{definition}

\theoremstyle{remark}
\newtheorem{remark}[theorem]{Remark}

\numberwithin{equation}{section}
\newcommand{\Fomega}{\til{\omega}}
\newcommand{\Falpha}{\til{\alpha}}
\newcommand{\til}{\tilde} %override
\newcommand{\pe}{\mathfrak}
\newcommand{\priPE}{\Pi}
\newcommand{\imPE}{\mathcal{I}}

\title[A Poincar\'e-Bendixson theorem for translation lines]{A Poincar\'e-Bendixson theorem for translation lines and applications to prime ends}
\author{Andres Koropecki}
\address{Universidade Federal Fluminense, Instituto de Matem\'atica e Estat\'\i stica, Rua Prof. Marcos Waldemar de Freitas Reis, s/n, Bloco H - Campus do Gragoat\'a, S\~ao Domingos, 24210-201, Niteroi, RJ, Brasil}
\email{ak@id.uff.br}
\author{Alejandro Passeggi}
\address{Universidad de la Rep\'ublica, CMAT, Igu\'a 4225, Montevideo, Uruguay} 
\email{alepasseggi@gmail.com}
\thanks{The first author was partially supported by FAPERJ-Brasil and CNPq-Brasil. The second author was supported by the group ``Sistemas din\'amicos'', No 618, C.S.I.C. UdelaR.}

\begin{document}
%\begin{abstract}
%\end{abstract}
\begin{abstract} For an orientation-preserving homeomorphism of the sphere, we prove that if a translation line does not accumulate in a fixed point, then it necessarily spirals towards a topological attractor. This is in analogy with the description of flow lines given by Poincar\'e-Bendixson theorem. We then apply this result to the study of invariant continua without fixed points, in particular to circloids and boundaries of simply connected open sets. Among the applications, we show that if the prime ends rotation number of such an open set $U$ vanishes, then either there is a fixed point in the boundary, or the boundary of $U$ is contained in the basin of a finite family of topological ``rotational'' attractors. This description strongly improves a previous result by Cartwright and Littlewood, by passing from the prime ends compactification to the ambient space. Moreover, the dynamics in a neighborhood of the boundary is semiconjugate to a very simple model dynamics on a planar graph. Other applications involve the decomposability of invariant continua, and realization of rotation numbers by periodic points on circloids.
\end{abstract}
\maketitle

\section{Introduction}

\subsection{A Poincar\'e-Bendixson theorem for translation lines}

The Poincar\'e-Bendixson theorem states that the $\omega$-limit of a regular orbit $\Gamma$ of a flow in $\SS^2$ either contains a fixed point or is a closed orbit which is attracting from the side that is accumulated by the given orbit. One of our main results states that a similar property holds if $\Gamma$ is only assumed to be a translation line, \ie the image of an injective continuous map $\gamma\colon \R\to \SS^2$ such that $f(\Gamma)=\Gamma$ and $\Gamma$ has no fixed points (which implies that $f|_\Gamma$ is topologically conjugate to the translation $x\mapsto x+1$ of $\R$ if $\Gamma$ is endowed with the linear topology induced by $\gamma$).  We always assume that $\gamma$ is parametrized so that the orientation matches the one induced by the dynamics. 
Translation lines appear naturally in the basins of connected topological attractors or repellors, or as stable and unstable branches of hyperbolic saddles.

The $\omega$-limit of $\Gamma$ is defined as $\omega(\Gamma)=\bigcap_{t>0} \ol{\gamma([t,\infty))}$, and if $\Gamma$ is disjoint from $\omega(\Gamma)$ its \emph{filled $\omega$-limit} $\Fomega(\Gamma)$ is the complement of the connected component of $\R^2\sm \omega(\Gamma)$ which contains $\Gamma$. The set $\Fomega(\Gamma)$ is a non-separating continuum, and we say that it is a \emph{rotational attractor} if it is a topological attractor and its external prime ends rotation number is nonzero (see \S\ref{sec:continua-attractors} for details). 

\begin{theoremain}\label{th:pb} Suppose that $\Gamma$ is an translation line for an orientation-preserving homeomorphism of $\SS^2$. Then $\omega(\Gamma)$ either contains a fixed point, or $\Gamma$ is a topologically embedded line and $\til{\omega}(\Gamma)$ is a rotational attractor disjoint from $\Gamma$.
\end{theoremain}

In the case that $\til{\omega}(\Gamma)$ is a rotational attractor, one can show that $\omega(\Gamma)$ has exactly two invariant complementary components (see Corollary \ref{coro:rotational-index}), which allow us to define two ``sides'' of $\omega(\Gamma)$. The theorem implies that $\omega(\Gamma)$ is topologically attracting from the side that is accumulated by $\Gamma$, and $\Gamma$ ``spirals'' towards $\omega(\Gamma)$; see Figure \ref{fig:horseshoe_rotat}. This is exactly what happens in the setting of the classical Poincar\'e-Bendixson theorem for flows. The usual proof of the Poincar\'e-Bendixson theorem is not easily adapted to this setting, due to the absence of flow-boxes. In fact, we emphasize that the line $\Gamma$ could accumulate (or self-accumulate) in a very intricate way. For instance, $\Gamma$ could be a branch of the unstable manifold of a hyperbolic saddle point with a homoclinic intersection (see Figure \ref{fig:horseshoe_rotat}). Even when $\Gamma$ is emebedded, its $\omega$-limit could be a \emph{wada lakes} continuum (a continuum with three complementary components and which is equal to the boundary of each component).

\begin{figure}[ht!]
\includegraphics[height=5cm]{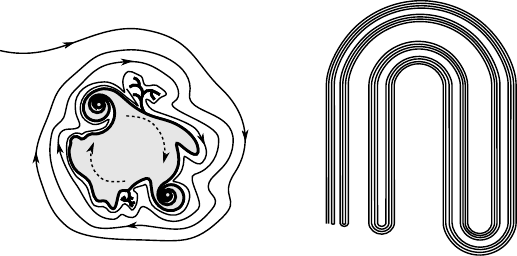} 
\caption{A rotational attractor, and a complicated translation line.}
\label{fig:horseshoe_rotat}
\end{figure}

Let us give a simple application. We say that a fixed point $p$ of a homeomorphism $f$ is a \emph{saddle} if $f$ is locally conjugate to $(x,y)\mapsto (\lambda_1 x, \lambda_2 y)$ at $p$, where $0<\lambda_1<1<\lambda_2$ (which is for instance the case whenever $p$ is a fixed point of diffeomorphism $f$ such that $Df(p)$ has eigenvalues $\lambda_1, \lambda_2$). An unstable branch of $p$ is a connected component of $W^u(p)\sm \{p\}$.

\begin{corollarymain} If $\Gamma$ is an unstable branch of a saddle fixed point of the homeomorphism $f\colon \SS^2\to \SS^2$, then either the $\omega$-limit of $\Gamma$ contains a fixed point or its filled $\omega$-limit is a rotational attractor disjoint from $\Gamma$.
\end{corollarymain}

A similar statement can be made for stable branches, using $f^{-1}$ instead of $f$.

\subsection{Invariant disks and prime ends}
Fix an orientation-preserving homeomorphism $f\colon \SS^2\to \SS^2$. If $U$ is an open invariant topological disk, the Carath\'eodory prime ends compactification provides a way of regularizing the boundary of $U$ by embedding $U$ in a closed unit disk, so that its boundary becomes a circle. This ``ideal'' circle is called the circle of prime ends of $U$. The map $f$ induces a homeomorphism $\hat f$ of the circle of prime ends, and understanding its relation with the dynamics of $f$ in the real boundary of $U$ is a subtle problem that has been studied in a number of works (for example \cite{cartwright-littlewood, walker,  barge-gillette, alligood-yorke, ortega-portal, corbato-portal, kln}). The rotation number of $\hat f$ is called the prime ends rotation number of $f$ in $U$, denoted $\rho(f, U)$, and one particularly relevant question is to what extent the Poincar\'e theory for circle homeomorphisms extends to this new invariant: what can one say about the dynamics of $f$ in $\bd U$, knowing its rotation number?

Unlike the classical case of the circle, it is possible to have $\rho(f,U)=0$ but no fixed point in $\bd U$ (see Figure \ref{fig:examples-pulpo}); however, the known examples where no such point exist are very particular. There are several known restrictions that such a map must satisfy \cite{cartwright-littlewood, alligood-yorke, kln}; for instance it cannot be area-preserving. In fact, it is known that any such example must have a very specific dynamics in the prime ends compactifications: namely, in a neighborhood of the circle of prime ends, it consists of the basins of a finite number of attracting and repelling prime ends \cite{cartwright-littlewood, matsumoto-nakayama}. Although this provides relevant information about the dynamics of $f|_U$ near the $\bd U$, it gives no information about the dynamics outside of $U$ (in particular, in $\bd U$). From our main results we are able to provide a description of the dynamics in an actual neighborhood of $\bd U$ in $\SS^2$:

To give a precise statement, let us say that a translation strip $T$ is the image of a continuous injective map $\phi\colon \R\times [0,1] \to \SS^2$ such that $\phi^{-1} f \phi(x,y) = (x+1,y)$. One may define the limit sets of translation strips and the notion of ``spiraling'' similarly to how it is done with translation lines (see Sections \ref{sec:translation-lines} and \ref{sec:spiral} for details).

\begin{theoremain}\label{th:disks} Suppose that $U\subset \SS^2$ is an open invariant topological disk such that its prime ends rotation number is $0$ but there are no fixed points in $\bd U$. Then $\bd U$ is contained in the union of the basins of a finite family of pairwise disjoint rotational attractors and repellors (at least one of each), disjoint from $U$ and with boundary in $\bd U$. 

More precisely, there exists $k\geq 0$ such that the fixed point index of $f$ in $U$ is $-k$ and finite families $\mc{A}, \mc{R}, \mc{T}$ of rotational attractors, rotational repellors, and translation strips (respectively) such that elements of $\mc{A}\cup \mc{R}\cup \mc{T}$ are pairwise disjoint and:
\begin{itemize}
\item Each element of $\mc{A}\cup \mc{R}$ is disjoint from $U$ and has boundary in $\bd U$;
\item Both $\mc{A}$ and $\mc{R}$ are nonempty and have at most $k+1$ elements;
\item $\mc T$ is nonempty and has at least $k+1$ and at most $2k+2$ elements;
\item Each element of $\mc{T}$ spirals from an element of $\mc{R}$ to an element of $\mc{A}$, and every element of $\mc{A}\cup \mc{R}$ appears in this way;
\item  The interiors of elements of $\mc T$ cover the set $\bd U \sm \bigcup_{K\in \mc{A}\cup \mc{R}} K$.
\end{itemize}
Moreover, the boundaries of elements of $\mc A \cup \mc R$ are the principal sets of fixed prime ends of $U$.
\end{theoremain}
See Figure \ref{fig:examples-pulpo} for two examples of possible situations where $k=1$.

The proof relies in a more local version of this result, which states that if an invariant open topological disk has a fixed prime end $\pe p$, then either there is a fixed point in the impression of $\pe p$, or the principal set of $\pe p$ is the boundary of a rotational attractor or repellor disjoint from $U$ (see Section \ref{sec:prime-ends} for precise definitions and Theorem \ref{th:prime-end} for a statement of the result).

\begin{figure}[ht!]
\includegraphics[height=4.5cm]{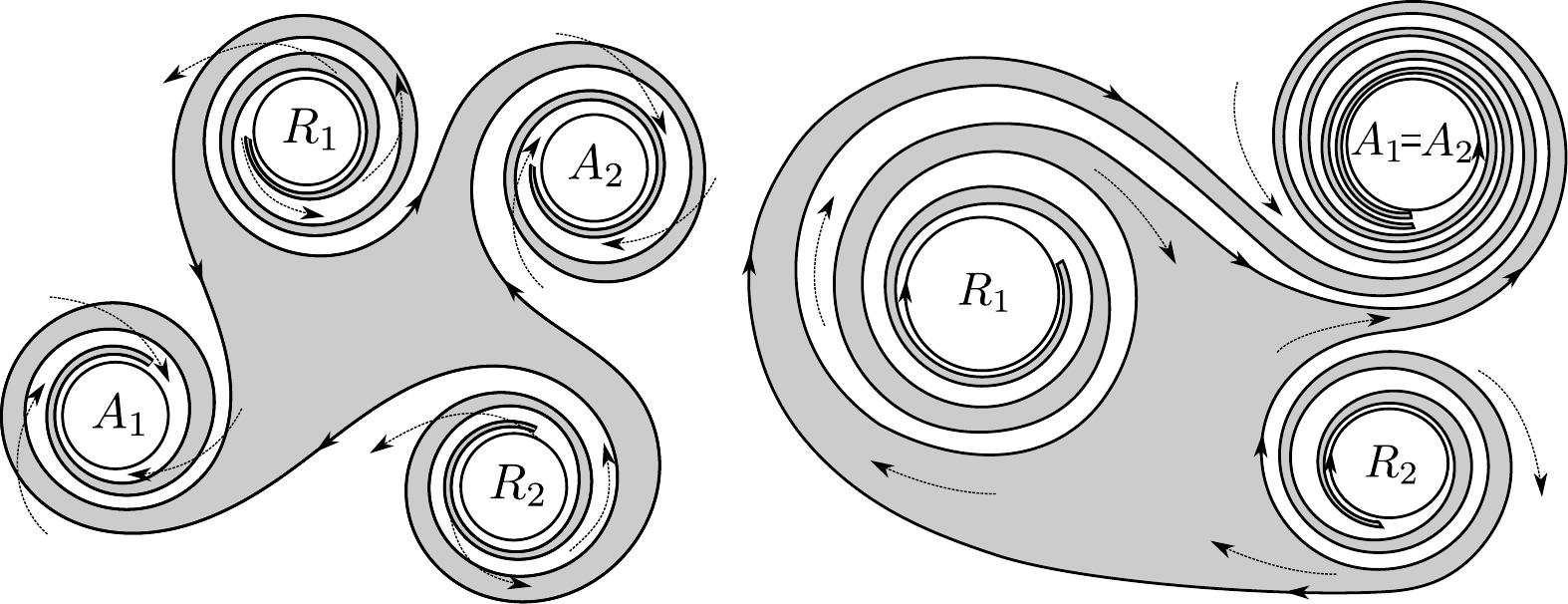} 
\caption{Two disks as described in Theorem \ref{th:disks}}
\label{fig:examples-pulpo}
\end{figure}

The condition of having a rotation number $0$ in Theorem \ref{th:disks} may be replaced by the equivalent condition of the disk having a fixed point index different from $1$ (see Corollary \ref{coro:disk-index}). Moreover, in that case the index is nonpositive.

One of the implications of the previous result is that the dynamics in $\bd U$ is topologically semi-conjugate to a very simple dynamics in an invariant graph: one may collapse the attractors/repellors to points (see Figure \ref{fig:collapse} ahead) and the translation strips to translation lines, obtaining a finite simple graph where  every point flows along an edge from a repellor to an attractor. This is detailed in Section \ref{sec:graph}.

As a simple application, noting that the basin of an attractor or repellor cannot intersect the chain recurrent set,  Theorem \ref{th:disks} implies the following:
\begin{theoremain}\label{coro:continua} If an invariant continuum in the chain recurrent set of $f$ has no fixed points, then its complement has exactly two invariant connected components, each with fixed point index $1$.
\end{theoremain}
This improves the main result of \cite{matsumoto-nakayama}, which states that a minimal continuum on the sphere has exactly two invariant complementary components. 
Note that in the particular case that $f$ has no periodic points, one may conclude that every non-invariant connected component in the complement of the continuum is wandering.

We also note that Theorem \ref{th:disks} can be used to describe the dynamics of arbitrary invariant continua without fixed points by studying the invariant connected components of its complement. To avoid more technical statements we will not further pursue this matter.

\subsection{Fixed and periodic points on circloids}

If $f\colon \A=\T^1\times \R\to \A$ is a homeomorphism isotopic to the identity and $K\subset \A$ is an essential invariant continuum, one may define the rotation number of a point $z\in K$ associated to a lift $\til{f}\colon \til{A}\simeq \R^2\to \til{A}$ of $f$ as $\rho(\til{f},z) =\lim_{n\to \infty} (\til{f}^n(\til{z})-\til{z})_1/n$, where $\til{z}\in \pi^{-1}(z)$, if the limit exists. The rotation interval of $K$ is $$\ol{\rho}(\til{f},K) = [\inf_{z\in K} \rho(\til{f},z), \sup_{z\in K} \rho(\til{f}, z)].$$ 
%\xxx{maybe state on the sphere}

A natural question, inspired by the  Poincar\'e-Birkhoff theorem, is whether $0$ being in the rotation interval implies that $\til{f}$ has a fixed point. This is generally not the case, as simple examples with $K$ equal to a closed annulus show. However, the answer is positive if one adds some dynamical restriction (like area-preservation, or a curve intersection property); see for instance \cite{franks-gen}. A different approach is to add a topological restriction on $K$ that guarantees a positive answer; for instance Barge and Gillette showed that the result holds if $K$ is a \emph{cofrontier} (\ie it separates the annulus into exactly two components and $K$ is the common boundary of these components) \cite{barge-gillette}. This was generalized in \cite{koro-circloid} to the case where $K$ is a \emph{circloid}, which means that $K$ is annular and essential (\ie it is essential in $\A$ and $A\sm K$ has exactly two components) and no proper subcontinuum of $K$ shares the same property.  Unlike cofrontiers, circloids may have nonempty interior. For example, the boundary of the grey disk on the left side of Figure \ref{fig:examples-pulpo}, together with the disks $A_1, A_2, R_1, R_2$ is a circloid with nonempty interior (we may remove a point form the grey disk so that the example lies in the annulus); however the boundary of the disk on the right side (joined with the three disks $A_1, R_1, R_2$) is not a circloid.

Circloids are particularly relevant because any annular continuum contains a circloid (and an invariant one if the given continuum is invariant). An arbitrary essential continuum needs not contain a circloid, but it always contains the boundary of a circloid. Hence it is relevant to know whether, given an invariant circloid $K$ which contains $0$ in its rotation interval, there exists a fixed point in the boundary of $K$. The answer is generally no; for instance this is what happens on the left side of Figure \ref{fig:examples-pulpo}). However, using our main results we can show that the only counterexamples are essentially as the aforementioned one:

\begin{theoremain}\label{th:circloid} Suppose that $K\subset \A$ is an $f$-invariant circloid with a fixed point. If there is no fixed point in $\bd K$, then there exists $k\geq 1$ such that there are $k$ rotational attractors and $k$ rotational repellors in $K$ which are pairwise disjoint and whose basins cover a neigbhorhood of $K$. Moreover, $f|_K$ is topologically semiconjugate to a Morse-Smale dynamics in the circle in the following way: there exists a monotone\footnote{\ie preimages of points are connected; in fact in this case they are non-separating continua.} map $h\colon \SS^2\to \SS^2$ and a homeomorphism $F\colon \SS^2\to \SS^2$ such that $hf = Fh$ and $h(K)$ is an $F$-invariant circle consisting of $k$ attracting and $k$ repelling (in $\SS^2$) fixed points.
\end{theoremain}

In fact, as in Theorem \ref{th:disks}, one may show that $K$ is contained in the union of the rotational attractors and repellors together with $2k$ translation strips, each of which spirals from a repellor to an attractor.

It is known that if the rotation set of a circloid has more than one element, then the boundary of the circloid is  indecomposable \cite{jager-koro} (see also \cite{barge-gillette}). In contrast, the previous theorem implies the following:

\begin{corollarymain}\label{coro:circloid-decomposable} If a circloid with indecomposable boundary has a fixed point in its interior, then its boundary has a fixed point.
\end{corollarymain}

We recall that for a rational $p/q$ in the rotation interval of $\til f$ is realized by a periodic point if there exists $x\in \R^2$ such that $\til f^q(x) = x + (p, 0)$ (which implies that $x$ projects to a period $q$ periodic point with rotation number $p/q$).
As a consequence of the previous result and the main result of \cite{jager-koro} we obtain the following:
\begin{theoremain}\label{th:reali} If the rotation interval of an invariant circloid is nonsingular, then every rational element of the rotation interval is realized by a periodic point in the boundary of the circloid.
\end{theoremain}

The result above is used in \cite{rata-pelado} in the proof that an attracting circloid with a nontrivial rotation interval has positive topological entropy.

\section{Notation and preliminaries}

Throughout this section $f\colon \SS^2\to \SS^2$ will always denote an orientation-preserving homeomorphism of the sphere.

\subsection{Disks, continua}
An open (closed) topological disk is a set homeomorphic to the open unit disk $\D$ (closed unit disk $\ol{\D}$) of $\R^2$. Any connected and simply connected open set in a surface is an open topological disk.
A \emph{continuum} is a compact connected set. A \emph{cellular continuum} is a continuum $K$ for which there exists a sequence of closed topological disks $(D_i)_{i\in \N}$ such that $D_{i+1}\subset \inter D_i$ and $K=\bigcap_{i\in \N} D_i$. In the sphere $\SS^2$ or the plane $\R^2$, cellular continua are precisely the non-separating continua.
A continuum $K$ is said to be \emph{decomposable} if it can be written as the union of two proper nonempty subcontinua.

\subsection{Lines, rays, limit sets}\label{sec:lines}

An (oriented) \emph{line} on a surface $S$ is the image $\Gamma$ of a continuous injective map $\gamma\colon \R\to S$, with the orientation induced by $\gamma$. Its $\omega$-limit set is $\omega(\Gamma) = \bigcap_{t\in \R} \ol{\gamma([t,\infty))}$, and its $\alpha$-limit set is $\alpha(\Gamma) = \bigcap_{t\in \R} \ol{\gamma((-\infty,t])}$. If $\Gamma$ is disjoint from $\omega(\Gamma)$, we say that $\Gamma$ is \emph{positively embedded}. Similarly, if $\Gamma$ is disjoint from $\alpha(\Gamma)$ we say that $\Gamma$ is \emph{negatively embedded}. An \emph{embedded} line is one which is both positively and negatively embedded (which is equivalent to saying that the map $\gamma$ is an embedding).

A \emph{ray} is a simple arc $\Gamma$ which is the image of a continuous injective map $\gamma\colon [0,\infty)\to S$. The point $\gamma(0)$ is called the initial point $\Gamma$. The $\omega$-limit set of $\Gamma$ is the set $\omega(\Gamma)= \bigcap_{t>0} \ol{\gamma([t,\infty))}$.
Note that $\omega(\Gamma)$ is always a continuum if $S$ is compact (or if $\Gamma$ is bounded). If $\Gamma$ is disjoint from $\omega(\Gamma)$, then $\Gamma$ is an \emph{embedded ray}, and in this case one has $\omega(\Gamma) = \ol{\Gamma}\sm \Gamma$.

If $\Gamma$ is an embedded ray in $\SS^2$, we may define its \emph{filled} $\omega$-limit as the set $\Fomega(\Gamma)$ which is the complement of the connected component of $\SS^2\sm \omega(\Gamma)$ containing $\Gamma$. Note that $\Fomega(\Gamma)$ is a cellular continuum and $\omega(\Gamma)= \bd\Fomega(\Gamma)$.
If $\Gamma$ is an embedded line in $\SS^2$, we may define the filled $\omega$-limit and $\alpha$-limit sets $\Fomega(\Gamma)$ and $\Falpha(\Gamma)$ similarly.

\subsection{Translation lines, rays and strips} \label{sec:translation-lines}
We say that a ray $\Gamma$ is a \emph{translation ray} for $f$ if $f(\Gamma)\subset \Gamma$ and there are no fixed points in $\Gamma$. Note that this implies that $\bigcap_{n\geq 0} f^n(\Gamma)=\emptyset$, so every $x\in \Gamma$ satisfies $\omega(x,f)\subset \omega(\Gamma)$. Similarly, a \emph{negative translation ray} is a translation ray for $f^{-1}$.
Note that $\omega(\Gamma)$ is $f$-invariant, and so is $\Fomega(\Gamma)$ if $\Gamma$ is embedded.

A translation line is an oriented line $\Gamma$ such that $f(\Gamma)=\Gamma$ and there are no fixed points in $\Gamma$. This is equivalent to saying that the dynamics induced by $f$ on $\Gamma$ is topologically conjugate to a translation $x\mapsto x+1$ of $\R$ (using in $\Gamma$ the topology induced by the immersion of $\R$).
A fundamental domain of the translation line $\Gamma$ is a simple subarc $\Gamma_0$ of $\Gamma$ joining a point $x\in \Gamma$ to $f(x)$. This implies that $f(\Gamma_0)\cap \Gamma_0=\{f(x)\}$ and $f^k(\Gamma_0)\cap \Gamma_0=\emptyset$ if $k>1$. Moreover, $\Gamma=\bigcup_{k\in \Z} f^k(\Gamma_0)$. 

Given a reference point $x\in \Gamma$, there are two rays $\Gamma^-_x$ and $\Gamma^+_x$ with initial point $x$ such that $\Gamma^+$ is a positive translation ray, $\Gamma^-$ is a negative translation ray, $\Gamma = \Gamma^-_x\cup \Gamma^+_x$ and $\Gamma^-_x\cap \Gamma^+_x=\{x_0\}$. 
We define the $\omega$-limit set of the translation line $\Gamma$ as $\omega(\Gamma)=\omega(\Gamma^+_x)$, and the $\alpha$-limit set as $\alpha(\Gamma)=\omega(\Gamma^-_x)$. This does not depend of the choice of $x$. 
If $\Gamma$ is disjoint from $\omega(\Gamma)$ (or $\alpha(\Gamma)$), we may also define the filled $\omega$-limit (or $\alpha$-limit) set $\Fomega(\Gamma) = \Fomega(\Gamma^+)$ (or $\Falpha(\Gamma)=\Fomega(\Gamma^-)$).

The translation line $\Gamma$ is embedded if and only if it is disjoint from $\omega(\Gamma)\cup \alpha(\Gamma)$. Note however that translation lines need not be embedded in general; for instance consider a branch of the stable manifold of a hyperbolic saddle exhibiting a transverse homoclinic intersection.

A parametrized \emph{strip} in a surface $S$ is a continuous injective map $\phi \colon \R\times [0,1]\to S$. % with the property that the interior of its image, $\phi(\R\times (0,1))$, is an arcwise connected component of the complement of $\phi(\R\times \{0\}) \cup \phi(\R\times \{1\})$. 
Two parametrized strips $\phi, \psi$ are equivalent if $\phi = \psi\circ h$ for some homeomorphism $h\colon \R\times[0,1]\to \R\times[0,1]$ preserving topological ends. An (oriented) \emph{strip} $T$ is an equivalence class of parametrized strips. We abuse the notation and use $T$ to refer to both a strip and its image (\ie the image of any parametrization of $T$). 

We say that $T$ is an \emph{embedded strip} if some parametrization of $T$ is a topological embedding. The $\omega$ and $\alpha$ limits of $T$ are defined as 
$$\omega(T) = \bigcap_{t\in \R} \phi([t, \infty)\times [0,1]),\quad  \alpha(T) = \bigcap_{t\in \R} \phi((-\infty, t])\times [0,1]).$$  
Note that if $T$ is embedded in $\SS^2$, then it is disjoint from $\alpha(T)$ and $\omega(T)$, and we may define the filled $\omega$ and $\alpha$ limit sets $\Fomega(T)$ and $\Falpha(T)$ as it was done with lines.

Finally, we say that $T$ is a (closed) \emph{translation strip} for $f$ if $T$ is $f$-invariant and admits a parametrization $\phi$ such that $\phi^{-1}f\phi(x,y) = (x+1, y)$ for all $(x,y)\in \R\times [0,1]$.

\subsection{Fixed point index} \label{sec:index}
Let us recall some facts about the fixed point index. See \cite{dold} for more details.
If $V\subset \R^2$ is an open set, $h\colon V\to \R^2$ is an orientation-preserving homeomorphism, and $D\subset V$ is a closed topological disk such that $\bd D\cap \fix(h) =\emptyset$, the \emph{fixed point index} $i(h, D)$ is the degree of the map $$(h-\id)|_{\bd D}\colon \bd D \to \R^2\sm \{0\},$$ where $\bd D$ is positively oriented. Note that if $D'\subset V$ is another closed topological disk whose boundary is homotopic to $\bd D$ in $V\sm \fix(h)$, then $i(h, D) = i(h, D')$.
If $p$ is an isolated fixed point of $h$, its index $i(f,p)$ is defined as $i(f, D)$ where $D$ is a sufficiently small disk containing $p$.

If $D\subset \SS^2$ is a closed topological disk such that $\bd D\cap \fix(f) =\emptyset$ and $f(D)\cup D\neq \SS^2$, we may define its fixed point index by choosing a point $\infty\in \SS^2\sm (D\cup f(D))$ and letting $i(f,D)$ be the fixed point index of $D\subset \SS^2\sm \{\infty\}$ for $f|_D\colon D\to \SS^2\sm \{\infty\} \simeq \R^2$.

If $V\subset \SS^2$ is an open set such that $\fix(f|_V)$ is compact, the fixed point index $i(f,V)$ is defined as follows: let $D_1,\dots, D_k$ be pairwise disjoint closed topological disks contained in $V$ such that $\bd D_i\cap \fix(f)=\emptyset$, $f(D_i)\cup D_i\neq \SS^2$, and $\fix(f|_V)\subset \bigcup_{i=1}^k D_i$. 
Then $$i(f,V) = \sum_{i=1}^k i(f, D_i).$$ 
This is independent of the choice of the disks $D_i$.
If $K\subset \SS^2$ is a compact set which has some neighborhood $V$ such that $\fix(f|_K) = \fix(f|_V)$, we may define its fixed point index as $i(f,K) = i(f,V)$, which is independent of the choice of $V$.

The fixed point index is additive in the following sense: if $V_1, V_2$ are open subsets of $\SS^2$ and $\fix(f|_{V_1})\cap \fix(f|_{V_2})=\emptyset$, then $i(f, V_1\cup V_2) = i(f, V_1) + i(f, V_2)$. A similar property holds for pairwise disjoint compact invariant sets.

We will use frequently the Lefschetz-Hopf theorem, which tells us that $i(f,\SS^2) = 2$ (recalling that $f$ preserves orientation).

\subsection{Prime ends}\label{sec:prime-ends}
We briefly introduce some notions about prime ends. For further details, see \cite{mather-cara, kln}.
Let $S$ be a surface and $U\subset S$ an relatively compact open topological disk. A \emph{cross-cut} of $U$ is an arc $\gamma\subset U$ joining two different points of $\bd U$ (but not including them). In this case $U\sm \gamma$ consists of exactly two connected components, called \emph{cross-sections} of $U$. A \emph{chain} of $U$ is a sequence $\mc{C} = (\gamma_i)_{i\in \N}$ of cross-cuts of $U$ such that $\gamma_{i+1}$ separates $\gamma_i$ from $\gamma_{i+2}$ in $U$ for all $i\in \N$. 
We say that a cross-cut $\eta$ \emph{divides} the chain $\mc{C}$ if there exists $i_0$ such that the cross-cuts $\{\gamma_i : i\geq i_0\}$ all belong to the same connected component of $U\sm \eta$. A chain $\mc{C}' = \mc(\gamma_j')_{j\in \N}$ divides the chain $\mc{C}$ if $\gamma_j'$ divides $\mc{C}$ for each $j\in \N$. The two chains $\mc{C}'$ and $\mc{C}$ are said to be equivalent if $\mc{C}$ divides $\mc{C}'$ and vice-versa. This is an equivalence relation in the family of all chains of $U$. 

We say that the chain $\mc{C}$ is a \emph{prime chain} if whenever a chain $\mc{C}'$ divides $\mc{C}$ one also has that $\mc{C}$ divides $\mc{C}'$. The family of all prime chains of $U$ modulo equivalence is denoted by $\bdPE(U)$ and its elements are called \emph{prime ends} of $U$. The set $U\sqcup \bdPE(U)$ is denoted by $\cPE(U)$ and one may topologize it so that $\cPE(U)$ becomes homeomorphic to the closed unit disk $\ol{\D}$, by means of the topology generated by open subsets of $U$ together with subsets of $\cPE(U)$ consisting of all prime chains divided by a given cross-cut (modulo equivalence). The space $\cPE(U)\simeq \ol{\D}$ is called the \emph{prime ends compactification} of $U$ and $\bdPE(U)\simeq \SS^1$ is the \emph{circle of prime ends}.

Every prime end $\mathfrak p\in \bdPE(U)$ can be represented by some chain $(\gamma_i)_{i\in \N}$ such that $\diam(\gamma_i)\to 0$ as $i\to \infty$. A \emph{principal point} of $\mathfrak p$, is a point $x\in S$ for which there exists such a representative chain with the property that $\gamma_i\to x$ in $S$ (\ie $\gamma_i$ is contained in any neighborhood of $x$ if $i$ is large enough). The set of all principal points is the \emph{principal set} of $\mathfrak{p}$, denoted by $\priPE(\mathfrak{p})$.

The \emph{impression} of $\mathfrak{p}$ the set $\imPE(\mathfrak p)$ of all points of $S$ which are accumulated by some sequence of points $(x_i)_{i\in \N}$ of $U$ such that $x_i\to \mathfrak{p}$ in $\cPE(U)$ as $i\to \infty$. If $(\gamma_i)_{i\in \N}$ is any prime chain representing $\mathfrak{p}$ and $D_i$ denotes the component of $U\sm \gamma_i$ containing $\gamma_j$ for all $j>i$, then one has $\imPE(\mathfrak p) = \bigcap_{i\in \N} \ol{D}_j \cap \bd U$.
The following result will be useful (see \cite{mather-cara}):

\begin{proposition}\label{pro:pri-im} If $\Gamma\subset U$ is a ray whose $\omega$-limit in $\cPE(U)$ consists of a single prime end $\mathfrak p\in \bdPE(U)$, then 
$$\priPE(\mathfrak p) \subset \omega(\Gamma) \subset \imPE(\mathfrak{p}),$$
where $\omega(\Gamma)$ denotes the $\omega$-limit in $S$ of $\Gamma$.
Moreover, there exists one such ray for which $\priPE(\mathfrak p) = \omega(\Gamma)$, and another one for which  $\imPE(\mathfrak p) = \omega(\Gamma)$.
In particular, the principal set and the impression of every prime end are subcontinua of $\bd U$.
\end{proposition}
Note that this means that the principal set is the smallest possible $\omega$-limit of such a ray $\Gamma$, and the impression is the largest one.

Let us also state a result for future reference:
\begin{lemma}\label{lem:ac-PE-bd} If $U\subset \SS^2$ is an open topological disk $\Sigma$ is a ray in $U$ which accumulates in every prime end in $\cPE(U)$, then it accumulates in every point of $\bd U$ in $\SS^2$, \ie $\bd(U)\subset \omega(\Sigma)$.
\end{lemma}
\begin{proof} It suffices to note that such a ray $\Sigma$ must intersect every cross-cut of $U$ whose diameter is small enough, and as one may easily verify every point of $\bd U$ is accumulated by cross-cuts of arbitrarily small diameter.
\end{proof}

\subsection{Spiraling lines and strips}\label{sec:spiral}
We say that a ray $\Gamma \subset \SS^2$ \emph{spirals} towards a continuum $K$ if $K$ has more than one point, $\Gamma$ is disjoint from $K$ and the following property holds: if $U$ is the connected component of $\SS^2\sm K$ containing $\Gamma$, identifying its prime ends compactification $\cPE(U)$ with $\ol{\D}$ by a homeomorphism $\phi\colon \cPE(U)\to \ol\D$, and letting $t\mapsto (r(t), \theta(t))$ represent a parametrization of $\Gamma$ in polar coordinates, one has 
$$\theta(t)\to\infty \text{ or } -\infty, \text{ and } r(t)\to 1$$
as $t\to \infty$. In other words, $\Gamma$ spirals towards the circle of prime ends in $U$.

When this happens, $\Gamma$ is necessarily an embedded line, and it is easy to show that $\omega(\Gamma) = \bd U$. Indeed, the definition implies that $\Gamma$ intersects every cross-cut of $U$, from which one deduces that $\Gamma$ accumulates in every point of $\bd U$ (for instance using the fact that accessible points are dense in $\bd U$; see \cite{mather-cara}).

As a consequence, if $\Gamma$ spirals towards $K$, then $\Gamma$ also spirals towards $\omega(\Gamma)$, hence saying that the ray $\Gamma$ disjoint from $K$ spirals towards a continuum $K$ is equivalent to saying that $\Gamma$ spirals towards its own $\omega$-limit and $\omega(\Gamma)\subset K\subset \Fomega(\Gamma)$.

We remark that two rays converging towards a single point are topologically indistinguishable (\ie there is a global homeomorphism mapping the closure of one to the closure of the other), which is why the definition of spiraling is only meaningful if the $\omega$-limit of the ray has more than one point.

Finally, similar definitions can be made for strips. We say that the strip $T$ disjoint from the continuum $K$ \emph{spirals towards} $K$ if $K$ has more than one point and the following property holds: letting $U$ be the connected component of $\SS^2\sm K$ containing $T$, given a parametrization $[0,\infty)\times [0,1] \to \ol{\D}$, $(x,y) \mapsto (r(x,y), \theta(x,y))$ of $\Gamma$ in polar coordinates, one has 
$$\theta(x,y) \to \infty \text{ or $-\infty$} \text{ and } r(x,y) \to 0 \text{ as $x\to \infty$ }$$
uniformly on $y$.
When this happens, one may verify that the $\omega$-limit and filled $\omega$-limit of $T$ coincide with corresponding limits of its two boundary lines (and of any other line properly embedded in $T$).

The following simple lemmas will be useful in our proofs.
\begin{lemma}\label{lem:spiral-omega} If $\Gamma$ is an embedded ray that spirals towards $\omega(\Gamma)$ and $\Gamma'$ is a ray disjoint from $\Gamma\cup \Fomega(\Gamma)$ such that $\omega(\Gamma')\cap \omega(\Gamma)\neq \emptyset$, then $\omega(\Gamma)\subset \omega(\Gamma')$. If in addition $\Gamma'$ is embedded, then $\Fomega(\Gamma)\subset \Fomega(\Gamma')$.
\end{lemma}
\begin{proof}
If $U = \SS^2\sm \Fomega(\Gamma)$, then $\Gamma'$ is a ray in $U$ which has accumulation points in $\bd U$. Considering $\Gamma$ and $\Gamma'$ as subsets of $\cPE(U)$, which we identify with $\ol{\D}$, we have that $\Gamma'$ has accumulation points in $\SS^1$. Using the facts that $\Gamma$ spirals towards $\SS^1=\bd \ol\D$ and is disjoint from $\Gamma'$, it is easy to verify that $\Gamma'$ must accumulate in all of $\SS^1$. We conclude from Lemma \ref{lem:ac-PE-bd} that, as a subset of $\SS^2$, the ray $\Gamma'$ accumulates in all of $\bd U=\omega(\Gamma)$.

If $\Gamma'$ is embedded, then since $\SS^2\sm \omega(\Gamma') \subset \SS^2\sm \omega(\Gamma)$ and $U$ is a connected component of the latter set containing $\Gamma'$, the connected component of $\SS^2\sm \omega(\Gamma')$ containing $\Gamma'$ is a subset of $U$. Thus $\SS^2\sm \Fomega(\Gamma')\subset \SS^2\sm \Fomega(\Gamma)$ and our claim follows.
\end{proof}

\begin{lemma}\label{lem:spiral} If $\Gamma$ is an embedded ray that spirals towards $\omega(\Gamma)$, then:
\begin{itemize}
\item[(1)] Every compact arc intersecting both $\SS^2\sm \Fomega(\Gamma)$ and $\Fomega(\Gamma)$ also intersects $\Gamma$;
\item[(2)] If an embedded ray $\Gamma'$ is disjoint from $\Gamma$, then one of the following holds:
\begin{itemize}
\item[(a)] $\Gamma'\subset \Fomega(\Gamma)$ or $\Gamma\subset \Fomega(\Gamma')$;
\item[(b)] $\Fomega(\Gamma)$ is disjoint from $\Fomega(\Gamma')$;
\item[(c)] $\Fomega(\Gamma) \subset \Fomega(\Gamma')$ and $\omega(\Gamma)\subset \omega(\Gamma')$.
\end{itemize}
\end{itemize}
\end{lemma}
\begin{proof}
The first item is straightforward and is left to the reader. 
To prove (2), note first that if $\Gamma'$ intersects $\Fomega(\Gamma)$ then (1) implies $\Gamma'\subset \Fomega(\Gamma)$, so (a) holds. Thus we may assume that $\Gamma'$ is disjoint from $\Fomega(\Gamma)$, and therefore $\Gamma'$ is disjoint from $\Gamma\cup \Fomega(\Gamma)$. If $\omega(\Gamma')$ intersects $\omega(\Gamma)$, we conclude from Lemma \ref{lem:spiral-omega} that (c) holds.
On the other hand if $\omega(\Gamma')$ is disjoint from $\omega(\Gamma)$, then $\bd \Fomega(\Gamma)$ and $\bd \Fomega(\Gamma')$ are two disjoint connected sets, and this easily implies that either $\bd \Fomega(\Gamma')\subset \inter \Fomega(\Gamma)$, or $\bd \Fomega(\Gamma)\subset \inter \Fomega(\Gamma')$, or $\Fomega(\Gamma)\cap \Fomega(\Gamma')=\emptyset$. The latter case means that (b) holds, while the first case is not possible since it implies that $\Gamma'$ intersects $\Fomega(\Gamma)$ contradicting our assumption. In the remaining case, $\bd \Fomega(\Gamma) \subset \inter \Fomega(\Gamma')$, noting that $\Gamma$ is disjoint from $\bd \Fomega(\Gamma') = \omega(\Gamma')$ (due to our assumption that $\omega(\Gamma)\cap \omega(\Gamma')=\emptyset$) we have that $\Gamma\cup \bd \Fomega(\Gamma) = \Gamma\cup\omega(\Gamma)$ is a continuum intersecting $\Fomega(\Gamma')$ and disjoint from $\bd\Fomega(\Gamma')$; therefore $\Gamma\cup\bd \Fomega(\Gamma)\subset \Fomega(\Gamma')$ and (a) holds.
\end{proof}

\subsection{Prime ends dynamics} \label{sec:prime-ends-dyn}
If $h\colon S\to S$ is an orientation-preserving homeomorphism and $U$ is an open $h$-invariant topological disk, then $h|_U$ extends to a homeomorphism $\hat{h}\colon \cPE(U)\to \cPE(U)$. Identifying $\cPE(U)$ with $\ol{\D} = \D\sqcup \SS^1$, we may define the \emph{prime ends rotation number of $h$ in $U$} as
$$\rho(h,U) = \rho(\hat{h}|_{\SS^1}) \in \SS^1,$$
which is the usual Poincar\'e rotation number of an orientation-preserving homeomorphism of the circle. 

If $K\subset \SS^2$ is an $f$-invariant cellular continuum, then we define its \emph{(exterior) prime ends rotation number} as $\rho_e(h,K) = \rho(h, \SS^2\sm K)$, noting that $\SS^2\sm K$ is an open topological disk.

\begin{lemma}\label{lem:rot-index} If $U\subset \SS^2$ is an open invariant topological disk and $\rho(f,U)\neq 0$, then the fixed point index of $f$ in $U$ is $1$. 
\end{lemma}
\begin{proof}
We can embed the prime ends compactification of $U$ into the closed unit disk of $\R^2$. Since the map induced on the circle of prime ends by $f|_U$ has nonzero rotation number, it has no fixed points, which implies that its degree is $1$. From this it follows easily that $i(f,U)=1$.
\end{proof}

In general a fixed prime end does not necessarily correspond to a fixed point in the boundary of $U$. However, the following result due to Cartwright and Littlewood says that this can only happen if the prime end is attracting or repelling (in the whole disk $\cPE(U)$).

\begin{lemma}[\cite{cartwright-littlewood}]\label{lem:C-L} If $U$ is an $f$-invariant open topological disk and $\pe p\in \bdPE(f)$ is a fixed prime end whose impression does not contain a fixed point of $f$, then $\pe p$ is either attracting or repelling in $\cPE(U)$.
\end{lemma}

As a consequence of this result one has the following (see \cite{matsumoto-nakayama} for a direct proof):
\begin{theorem}\label{th:matsu-naka} Suppose that $U$ is an open invariant topological disk, $\rho(f,U) = 0$, and $f$ has no fixed point in $\bd U$. Then $i(f, U) = -k < 0$, and the dynamics induced by $f$ on the circle of prime ends has exactly $2k+2$ fixed points, $k+1$ of which are attracting and $k+1$ of which are repelling. Moreover each attracting prime end is globally attracting in the prime ends compactification of $U$, and similarly for the repelling fixed points.
\end{theorem}

We remark that the claim about the index of $U$ is not explicitly stated in the aforementioned articles, but it follows from Lemma \ref{lem:index-LC} below.

From the previous theorem, and Lemma \ref{lem:rot-index} we have:

\begin{corollary}\label{coro:disk-index} If $U$ is an open invariant topological disk without fixed points in its boundary, then the fixed point index in $U$ is at most $1$. Moreover, the index is $1$ if and only if $\rho(f, U)\neq 0$.
\end{corollary}

\subsection{Attractor-Repellor graphs and index}

We first state a simple result which will often be used:
\begin{lemma}\label{lem:index-LC} Suppose $U$ is an open invariant topological disk whose boundary consists of a finite family $F$ of fixed points of $f$, each of which is either attracting or repelling in $\ol U$, together with a finite family $\mc F$ of (pairwise disjoint) translation lines, each connecting two elements of $F$. Assume further that $\bd U = \bd \SS^2\sm U$. Then $i(f,U) = 1-m/2 \leq 0$, where $m$ is the number of elements of $\mc F$ (which is necessarily even).
\end{lemma}
\begin{proof}
Note that a point of $F$ may be the $\omega$-limit (or $\alpha$-limit) of more than two elements of $\mc F$ (as in the disk on the right side in Figure \ref{fig:collapse} ahead). However, due to the fact that $\bd U$ is locally connected, one may identify $\cPE(U)$ with the closed unit disk $\ol{\D}$ in such a way that there exists a continuous surjection $\phi \colon \ol{\D} \to \ol{U}$ such that every point of $\ol U$ has a unique preimage except perhaps the elements of $F$ (see for instance \cite[Proposition 2.5]{pommerenke}). Thus $W = \bigcup_{\Gamma\in \mc F} \phi^{-1}(\Gamma)$ is a finite union of $m$ open intervals in $\SS^1$, where $m$ denotes the number of elements of $\mc F$. Moreover, it is easy to see that $\hat F = \SS^1\sm W$ has empty interior (or see \cite{beurling}), so it must consist of finitely many points which are the endpoints of the intervals in $W$, and $\hat F$ has exactly $m$ elements as well.

If $\hat f$ denotes the map induced on $\ol \D$ by $f|_U$, one easily verifies that each element of $F$ is either an attractor or a repellor in $\ol{\D}$, and there are no other fixed points in $\SS^1$ (since each component of $W$ is a translation interval). Note that $m\geq 2$ and it must be even (since every attractor is followed by a repellor in the cyclic order of $\SS^1$). A computation shows that $i(f,U) = i(\hat f, \D) = 1-m/2 \leq 0$: one may glue two copies of $\ol{\D}$ by identifying their boundary circles to obtain a homeomorphism of the sphere, where each fixed point of identified circles is either atracting or repelling. Noting that attracting and repelling fixed points have index $1$, one sees that $2i(\hat{f},\D) + m = 2$.
\end{proof}

By a \emph{planar graph} $G$ in $\SS^2$ we mean a finite set of vertices $V(G)$ (points) and edges $E(G)$ (lines) each connecting two vertices, such that the edges are pairwise disjoint and $G=V\cup E$.

Let us say that a planar graph $G\subset \SS^2$ is an attractor-repellor graph for a homeomorphism $F\colon \SS^2\to \SS^2$ if every vertex of $G$ is an attracting or repelling fixed point of $F$, every edge is a translation line joining a repellor to an attractor, and in addition every vertex has even degree. 
Note that in particular a neighborhood of $G$ is contained in the union of the basins of the attracting and repelling fixed points.

\begin{lemma} If every vertex of a planar graph $G$ has even degree, then every edge belongs to two different faces and in particular every face $D$ satisfies $\bd D = \bd \SS^2\sm D$.
\end{lemma}
\begin{proof} It suffices to prove the claim for each connected component of $G$, so we may assume that $G$ is connected. For a connected graph, by a well known theorem of Euler, every edge has even degree if and only if there is an Eulerian cycle (a loop in the graph going through all edges without repetition). The existence of such a cycle implies that for any edge $E$, the graph $G\sm E$ is connected (recalling that we do not include vertices in the edges). Thus the face of $G\sm E$ containing $E$ is an open topological disk $D$, and $E$ separates $D$ into two connected components $D_1, D_2$ which are two different faces of $G$ containing $E$. This proves the lemma.
\end{proof}

\begin{lemma}\label{lem:graph-index} Every face of an attractor-repellor graph $G$ has nonpositive fixed point index. More specifically, every face has index $2 -e/2 - c$ where $e$ is the number of edges of the face (necessarily even) and $c$ is the number of connected components of its boundary.
\end{lemma}
\begin{proof}
We first note that from the fact that every vertex has even degree, the previous Lemma implies that if $D$ is a face of $G$ then $\bd D = \bd \SS^2\sm D$.

For simply connected faces (\ie with $c=1$), Lemma \ref{lem:index-LC} implies that the fixed point index is $1-e/2$ and the claim follows. The case of multiple connected components is shown by induction: Assume the face $D$ has $n+1$ boundary components, and let $G_1$ be the connected component of $G$ containing one of the boundary components of $D$ and $G_n = G\sm G_1$. Then $G_1$ and $G_n$ are also attractor-repellor graphs. The face $D_1$ of $G_1$ containing $D$ has index $1-e_1/2$, where $e_1$ is the number of edges of $G_1$, which means that $\SS^2\sm D_1$ has index $2-(1-e_1/2) = 1+e_1/2$. On the other hand by induction, the face $D_n$ of $G_n$ containing $D$ has index $2-e_n/2 - n$. One easily verifies that $D = D_n\cap D_1 = D_n\sm (\SS^2\sm D_1)$, which implies that the index of $D$ is $2-e_n/2-n - (1+e_1/2) = 2 - e/2 - (n+1)$, proving the induction step.
\end{proof}

\subsection{Brouwer theory}\label{sec:brouwer}

We recall some classical results from Brouwer theory. 

\begin{lemma}\label{lem:free-homeo} Suppose $h\colon \R^2\to \R^2$ is an orientation-preserving homeomorphism without fixed points. If $K\subset \R^2$ is a compact connected set such that $f(K)\cap K=\emptyset$ then $f^n(K)\cap K=\emptyset$ for all $n\neq 0$.
\end{lemma}
The previous lemma is a simple consequence of the next result about periodic disk chains. A set is called \emph{free} for a homeomorphism if it is disjoint from its image by the homeomorphism. A \emph{free disk chain} for an orientation-preserving planar homeomorphism $h\colon \R^2\to \R^2$ is a family $D_0,D_1,\dots, D_{n-1}$ of pairwise disjoint free topological disks such that for $0\leq i\leq n-1$ there exists $m_i>0$ with $f^{m_i}(D_i)\cap D_{i+1\, (\text{mod }n)}\neq \emptyset$.

\begin{lemma}[\cite{franks-gen}]\label{lem:franks} If an orientation-preserving planar homeomorphism has a free disk chain, then it has a fixed point.
\end{lemma}

As a well-known consequence of the previous result we have:
\begin{corollary}\label{coro:brouwer} If an orientation-preserving homeomorphism $h\colon\R^2\to \R^2$ has a nonwandering point, then it has a fixed point. In particular, if $h$ has no fixed points, then for any $z\in \R^2$ one has $||h^n(z)||\to \infty$ as $n\to \pm \infty$.
\end{corollary}

Note that in particular if $h$ has a compact invariant set, then it has a nonwandering point and therefore by the previous corollary it has a fixed point.

The next lemma says that if $\Gamma$ is a translation line for an orientation-preserving plane homeomorphism without fixed points and $D$ is a free disk, every fundamental domain of $\Gamma$ between two points of $\Gamma\cap D$ must also intersect $D$.

\begin{lemma}\label{lem:trans-fix} Let $h\colon \R^2\to \R^2$ be an orientation-preserving homeomorphism without fixed points, and $\Gamma$ a translation line for $h$. Let $D$ be a closed disk such that $h(D)\cap D=\emptyset$. Let $\alpha$ be any compact subarc of $\Gamma$ joining two points of $D$ and only intersecting $D$ at its endpoints. Then $h(\alpha)\cap \alpha = \emptyset$.
\end{lemma}

\begin{proof} Let $x,y\in D$ be the endpoints of $\alpha$, with $y$ after $x$ in the linear order of $\Gamma$. Suppose for a contradiction that $h(\alpha)\cap \alpha\neq \emptyset$, and let $\alpha_0$ be the fundamental domain between $x$ and $h(x)$. Note that $\alpha_0\subset \alpha$.  Let $k$ be the unique integer such that $h^{-k}(y)\in \alpha_0\sm\{x\}$. Since $\Gamma$ is a translation line, our choice of $x$ implies that $k\geq 0$, and since $h(x)\neq y$ (because $D$ is disjoint from its own image), we have $k\geq 1$. Let $\sigma$ be the compact subarc of $\alpha_0$ joining $h^{-k}(y)$ to $h(x)$ (allowing $\sigma =\{h(x)\}$ in the case that $h^{-k}(y)=h(x)$). Note that $\sigma$ is compact, disjoint from its own image, and also disjoint from $D$ since it does not contain $x$ or $y$. We may thus choose neighborhoods $V_0$ of $\sigma$ and $V_1$ of $D$ which are disjoint open topological disks such that $h(V_i)\cap V_i=\emptyset$ for $i\in \{0,1\}$ (see Figure \ref{fig:trans-brouwer}). 
\begin{figure}[ht!]
\includegraphics[height=5cm]{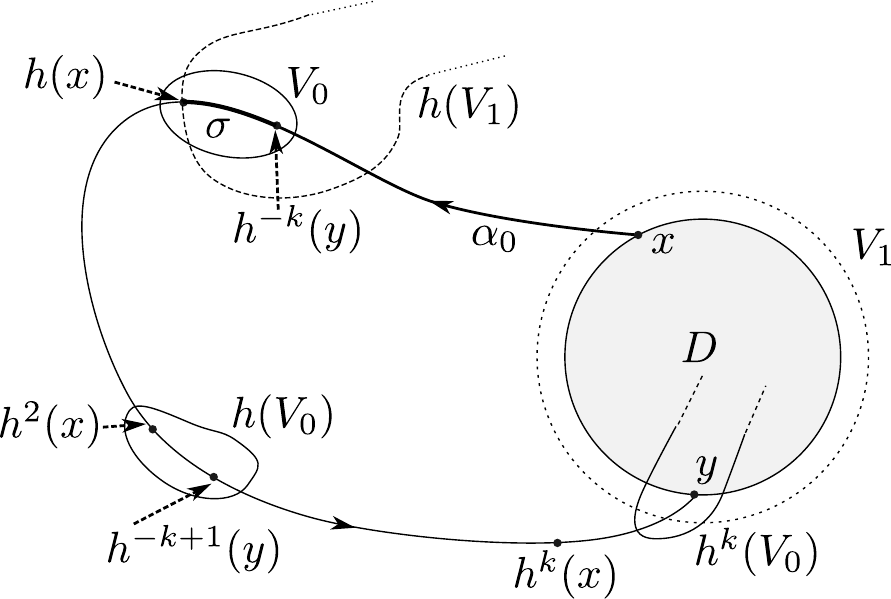}
\caption{The periodic disk chain in the proof of Lemma \ref{lem:trans-fix}}
\label{fig:trans-brouwer}
\end{figure}
Moreover, $h^k(V_0)\cap V_1\neq \emptyset$ and $h(V_1)\cap V_0\neq\emptyset$. This means that $V_0,V_1$ is a periodic disk chain (of period $2$), which by Lemma \ref{lem:franks} is not possible since $h$ has no fixed points. This contradiction proves the lemma.
\end{proof}

From the previous lemma one can deduce the following (see \cite[Proposition 3.6]{guillou}):
\begin{corollary}\label{coro:guillou} Let $h\colon \R^2\to \R^2$ be an orientation-preserving homeomorphism without fixed points. Then every translation line for $h$ is embedded.
\end{corollary}

\subsection{Continua, attractors}\label{sec:continua-attractors}

A \emph{continuum} is a compact connected set. A \emph{cellular} continuum is a proper subcontinuum of $\SS^2$ which is non-separating in $\SS^2$, which is equivalent to saying that its complement is an open topological disk. As an alternative (but equivalent) definition, a cellular continuum is the intersection of a decreasing chain of closed topological disks $(D_i)_{i\in \N}$ such that $D_{i+1}\subset \inter D_i$ for all $i\in \N$. By a theorem of Cartwright and Littlewood, every $f$-invariant cellular continuum has a fixed point \cite{cartwright-littlewood}. 

If $K\subset \SS^2$ is an $f$-invariant cellular continuum, $\SS^2\sm K$ is an invariant open topological disk. We define the \emph{external prime ends rotation number} of $K$ as $\rho_e(f,K) = -\rho(f,\SS^2\sm K)$ (the chain of sign is to preserve the notion that a positive rotation number corresponds to a counter-clockwise rotation). To avoid confusion, note that the change of sign is not the antipodal map, but the inversion in the lie group $\T^1$; so if $\rho(f,\SS^2\sm K) = r+\Z\in \T^1$ then $\rho_e(f,K) = -r+\Z$.

A \emph{trapping region} is an open set $V$ such that $f(\ol{V})\subset V$. The maximal invariant set $A$ in a trapping region $V$ is called an \emph{attractor}. Note that $A=\bigcap_{n\in \N} f^n(V)$, and $A$ is necessarily compact. Note that given any neighborhood $V_0$ of the attractor $A$, we can always find a trapping region for $A$ contained in $V_0$. 
The \emph{basin} of the attractor $A$ is the set $\bigcup_{n\in \N} f^{-n}(V)$, which is the set of all $x\in \SS^2$ such that $\omega(x,f)\subset A$. 

A \emph{repellor} $R$ is an attractor for $f^{-1}$, and the basin of a repellor is defined as the basin of $R$ as an attractor of $f^{-1}$, \ie the set of all $x\in \SS^2$ such that $\alpha(x,f)\subset R$. 

If $K\subset \SS^2$ is a cellular continuum, we say that $K$ is a \emph{rotational attractor} (or repellor) if it is an attractor (or repellor) and its external prime ends rotation number $\rho_e(f,K)$ is nonzero.

The following lemma is useful to find attractors; see \cite[Lemma 2.9]{shub}:
\begin{lemma}\label{lem:shub} If $K$ is a compact invariant set such that $K = \bigcap_{k\geq 0} f^k(\ol{U})$ for some open neighborhood $U$ of $K$, then $K$ is an attractor of $f$.
\end{lemma}

The following simple facts will be useful in our proofs:

\begin{lemma}\label{lem:attractor-intersect} Let $\Gamma$ be a translation line. If $\omega(\Gamma)\cap R\neq \emptyset$ for some repellor $R$, then $\Gamma\cap R\neq \emptyset$. Similarly if $\alpha(\Gamma)\cap A\neq \emptyset$ for some attractor $A$, then $\Gamma\cap A\neq \emptyset$.
\end{lemma}

\begin{proof} It suffices to prove the first claim. Suppose for a contradiction that there is a repellor $R$ such that $\omega(\Gamma)\cap R\neq \emptyset$ and $\Gamma\cap R=\emptyset$. Given a fundamental domain $\Gamma_0$ of $\Gamma$, since $\Gamma_0$ is compact and disjoint from $R$ we may choose a trapping region $V\subset \SS^2$ for $R$ (\ie $f^{-1}(\ol{V})\subset V$ and $R=\bigcap_{n\geq 0} f^{-n}(V)$) such that $V\cap \Gamma_0= \emptyset$. But since $\omega(\Gamma)$ intersects $R$, there exists $n>0$ such that $f^n(\Gamma_0)\cap V\neq \emptyset$. This implies that $\emptyset \neq \Gamma_0\cap f^{-n}(V) \subset \Gamma_0\cap V$, contradicting our choice of $V$.
\end{proof}

\section{Preliminary lemmas}

In this section we prove some general lemmas which are helpful in the proof of the main theorems. We fix, as before, an orientation-preserving homeomorphism $f\colon \SS^2\to \SS^2$.

\subsection{4-Branches lemma}

If the $\omega$-limit (or $\alpha$-limit) set of a translation line $\Gamma$ consists of a single (fixed) point $p$, we say that $\Gamma$ is a stable (or unstable) branch of $p$.  If $\Gamma_1,\dots, \Gamma_k$ are disjoint stable branches of a fixed point $p$, there is a well-defined cyclic order defined by choosing a positively oriented simple loop $\alpha$ around $p$ intersecting all branches and considering considering the first intersection of each branch with $\alpha$ (starting from $p$). This is independent of the choice of $\alpha$. 

\begin{lemma}[4-branches Lemma]\label{lem:4branch}
Suppose that $\Gamma_1^s, \Gamma_2^s$ are two stable branches of a fixed point $p$, $\Gamma_1^u, \Gamma_2^u$ are two unstable branches of $p$, the four branches are pairwise disjoint and they alternate in the cyclic order around $p$ (in the sense that no two stable or unstable branches are consecutive). Then any open topological disk intersecting all four branches intersects its own image by $f$.
\end{lemma}

\begin{corollary}\label{coro:4branch} Under the hypothesis of the previous lemma, any point accumulated by all four branches is fixed by $f$.
\end{corollary}

\begin{proof}[Proof of Lemma \ref{lem:4branch}]
Suppose that $D$ is a closed topological disk intersecting all four branches. If $D$ contains $p$ then the result holds trivially, so we may assume that $p\notin D$. Let $\gamma_1^s$ and $\gamma_2^s$ be the sub-arcs of $\Gamma_1$, $\Gamma_2$ from $p$ to the first intersection of the corresponding branch with $D$, and define $\gamma_1^u, \gamma_2^u$ similarly. Let $\sigma\subset D$ be an arc joining the endpoint of $\gamma_1^s$ to the endpoint of $\gamma_2^s$ and otherwise contained in the interior of $D$. %Choosing $\sigma$ adequately we may assume that it also intersects both $f(\gamma_1^u)$ and $f(\gamma_2^u)$. 
\begin{figure}[ht]
\includegraphics[height=4.5cm]{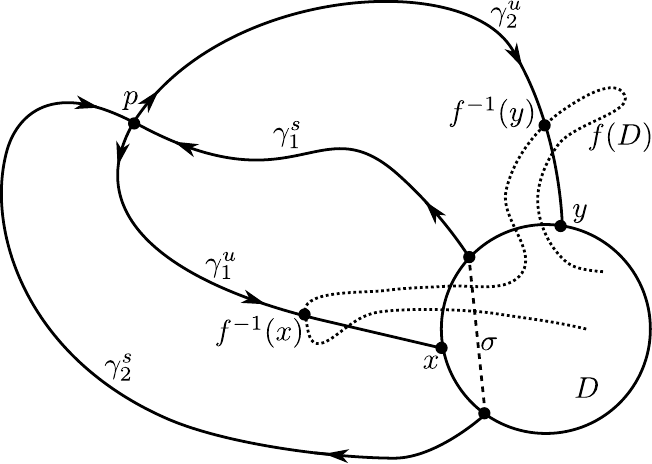}
\caption{Proof of Lemma \ref{lem:4branch}}
\label{fig:4branches}
\end{figure}
Because of the alternating cyclic order of the branches, the loop $\alpha$ formed by $\gamma_1^s \cup \sigma \cup \gamma_2^s\cup\{p\}$ separates $\gamma_1^u$ from $\gamma_2^u$ (see Figure \ref{fig:4branches}. This implies that $f^{-1}(D)$ intersects the two connected components of $\SS^2\sm \alpha$, and so it intersects $\alpha$. On the other hand $f^{-1}(D)$ is disjoint from $\gamma_1^s$ and $\gamma_2^s$ (otherwise one of these arcs would intersect $D$ other than at its endpoint). Thus $f^{-1}(D)$ intersects $\sigma$, and it follows that $f(D)$ intersects $D$. 
\end{proof}

\subsection{Stable and unstable branches of disks}\label{sec:stable-rays}

Suppose that $U$ is an open invariant topological disk of nonpositive index and $f|_{\bd U}$ has no fixed points. We know from Corollary \ref{coro:disk-index} that $\rho(f,U)=0$ and therefore by Theorem \ref{th:matsu-naka} the extension of $f|_U$ to $\cPE(U)$ has a Morse-Smale dynamics in a neighborhood of $\bdPE(U)$, with exactly $k$ attractors and $k$ repellors alternating in the cyclic order of $\bdPE(U)$ (where $1-k$ is the index of $U$).

We say that $\Gamma$ is an \emph{unstable branch} of $U$ (for the map $f$) if $\Gamma$ is a translation line for $f$ contained in $U$ and $\omega(\Gamma)\cap U=\emptyset$. %The latter is equivalent to saying that $\Gamma\subset U$ and $\omega(\Gamma)\subset \bd U$. 
If $I$ denotes the $\omega$-limit in $\cPE(U)$ of $\Gamma$, then $I$ consists of a single attracting prime end. Indeed, $I$ is a connected subset of $\bdPE(U)$ and cannot contain a point in the basin of a repelling fixed prime end. This is because any point of $U$ close enough to a point of $\bdPE(U)$ in the basin of a repelling fixed prime end has a pre-orbit contained in an arbitrarily small neighborhood of $\bdPE(U)$, whereas any point of $\Gamma$ has a pre-orbit in some fundamental domain $\Gamma_0\subset U$ of $\Gamma$. 
Thus $\Gamma$ has a unique $\omega$-limit point in $\cPE(U)$ which is some attracting prime end. Thus an unstable branch of $U$ can be seen as a stable branch of some attracting fixed prime end of $U$. See Figure \ref{fig:branches}.

\begin{figure}[ht!]
\includegraphics[height=3cm]{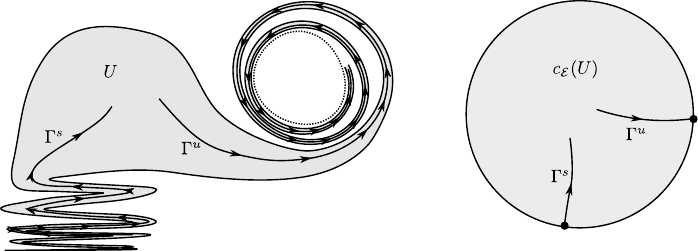}
\caption{A stable and unstable branch of $U$}
\label{fig:branches}
\end{figure}

A \emph{stable branch} $\Gamma$ of $U$ is an unstable branch of $U$ for the map $f^{-1}$, or equivalently $\Gamma$ is a translation line for $f$ in $U$ with its $\alpha$-limit set disjoint form $U$. This implies that the $\alpha$-limit of $\Gamma$ in $\cPE(U)$ consists of a single repelling fixed prime end, and we may regard $\Gamma$ as an unstable branch for this prime end. 

We have therefore:
\begin{proposition} If $\Gamma$ is a stable/unstable branch of an open invariant topological disk $U$ of nonpositive index without fixed points in $\bd U$, then $\Gamma$ is an unstable/stable branch of some repelling/attracting fixed prime end of $U$.
\end{proposition}

\subsection{A simplification for invariant disks}

In the setting of Theorem \ref{th:matsu-naka}, it will be useful to simplify the dynamics in $U$ by modifying our map in a compact subset of $U$ (thus leaving the boundary dynamics unmodified). This will help us, for instance, to apply Lemma \ref{lem:4branch} to topological disks. For this, we first introduce some ``model'' maps.

Fix an orientation-preserving homeomorphism $G_0\colon \ol{\D}\to \ol{\D}$ with exactly two fixed points, both in $\SS^1$, one of which is attracting and the other repelling, and such that every other point is in the basin of both the attractor and the repellor. For instance $G_0$ could be chosen to be a hyperbolic map of the Poincar\'e disk. Clearly any pair of homeomorphisms of this type are topologically conjugate.

For each $k\geq 1$, we fix a``model'' orientation-preserving homeomorphism $G_k\colon \ol{\D}\to \ol{\D}$ with the following properties:
\begin{itemize}
\item $G_k|_{\SS^1}$ has exactly $2k+2$ fixed points, $a_0,\dots, a_k$ and $r_0,\dots, r_k$, where each $a_i$ is an attractor in $\ol \D$ and $r_i$ is a repellor in $\ol \D$ (and they alternate in the cyclic ordering of $\SS^1$);
\item $G_k|_{\D}$ has exactly one fixed point $p$ with exactly $k+1$ stable branches, $\Gamma^s_0,\dots, \Gamma^s_k$ and  $k+1$ unstable branches $\Gamma^u_0,\dots, \Gamma^u_k$, and all these branches are pairwise disjoint;
\item $\omega(\Gamma^s_i)=\alpha(\Gamma^u_i) = \{p\}$, $\alpha(\Gamma^s_i)=\{r_i\}$, and $\omega(\Gamma^u_i)=a_i$;
\item every connected component $S$ of $\D\sm (\bigcup_{i=0}^k \Gamma_i^s\cup \Gamma_i^u)$ is contained in the basins of both the attractor and the repellor that lie in the boundary of $S$.
\end{itemize}
Such a map can be chosen as the time-1 map of a flow. See Figure \ref{fig:model}.

\begin{figure}[ht!]
\includegraphics[height=3cm]{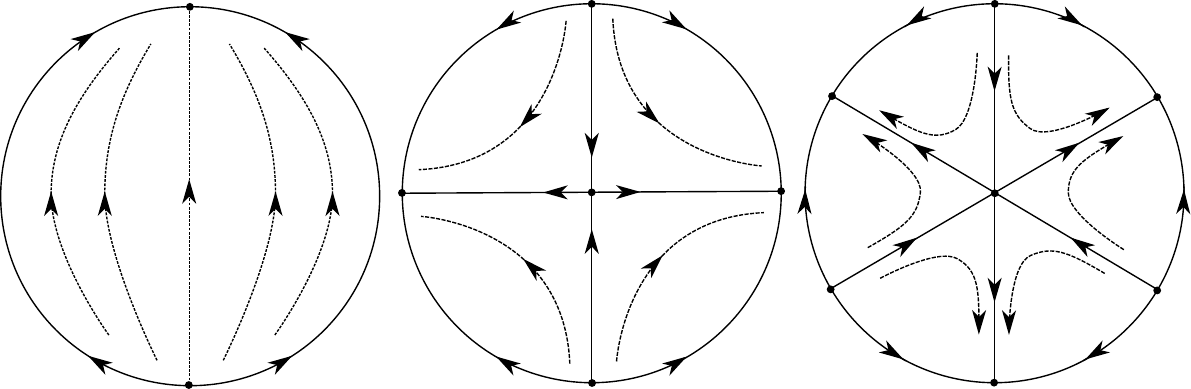}
\caption{The maps $G_k$ for $k=0,1,2$.}
\label{fig:model}
\end{figure}

We have the following:
\begin{lemma}\label{lem:model} Let $g\colon \ol{\D}\to \ol{\D}$ be an orientation-preserving homeomorphism such that $g|_{\bd \D}$ has exactly $2k$ fixed points, $k$ of which are attracting in $\ol{\D}$ and $k$ repelling in $\ol{\D}$, where $k\geq 1$. Then there exists a homeomorphism $g'\colon \ol{\D}\to \ol{\D}$ which coincides with $g$ in a neighborhood of $\bd \ol{\D}$ such that $g'$ is topologically conjugate to $G_{k-1}$.
\end{lemma}

\begin{proof}
We only sketch the proof since it uses routine arguments. First note that it suffices to find a homeomorphism which is topologically conjugate to $g$ and which coincides with $G_{k-1}$ in a neighborhood of $\bd \ol\D$. Indeed, if $h\colon \ol\D\to \ol \D$ is a homeomorphism and $hgh^{-1}$ coincides with $G_{k-1}$ in a neighborhood of $\bd \ol\D$, then $g' = h^{-1}G_{k-1}h$ has the required properties.

The general idea of the proof is to first make a conjugation of $g$ to obtain a map which coincides with $G_{k-1}$ in a neighborhood $E$ of $\SS^1 \sm \fix(g|_{\SS^1})$ bounded by translation lines which connect repellors to attractors cyclically (as the greyed out area in Figure \ref{fig:flower}). The region $E$ is not yet a neighborhood of $\bd \ol \D$ as it misses the fixed points, but then we make additional conjugations supported in a neighborhood of each fixed point if $g|_{\SS^1}$ to make $g$ coincide with $G_{k-1}$ (while keeping $g$ equal to $G_{k-1}$ in $E$). With this argument we obtain a (orientation-preserving) homeomorphism $h\colon \ol{\D}\to\ol{\D}$ such that $hgh^{-1}$ coincides with $G_{k-1}$ in a neighborhood of $\bd \ol{\D}$ as required.

We provide the details for the case $k=1$ to avoid cumbersome notation, but the general case follows the same steps. Let $p^+$ and $p^-$ denote the attracting and repelling fixed points in $\SS^1$. One may choose a neighborhood $V^-$ of $p^-$ in $\ol{\D}$ contained in the basin of $p^-$, bounded by a simple arc $\gamma^-$ (which joins two different points of $\SS^1\sm \{p^-, p^+\}$) and such that the closure of $g^{-1}(V^-)$ is contained in $V^-$. Note that $\SS^1\sm \{p^+, p^-\}$ consists of two lines $L_1$ and $L_2$ contained in the intersection of the basins of $p^+$ and $p^-$. Let $x_i$ be the endpoint of $\gamma^-$ in $L_i$, and let $I_i$ be the compact subarc of $L_i$ joining $x_i$ to $g(x_i)$. Since $I_i$ is in the basins of $p^\pm$, so is some neighborhood $W$ of $I_i$ (in $\ol{\D}$). One may choose an arc $\alpha_i$ contained in $g(V^-)\sm \ol{V^-}$ except for its endpoints $y_i \in \gamma_-$ and $g(y_i)\in g(\gamma^-)$.  Choosing $\alpha_i$ close enough to $I_i$ we may guarantee that, if $\sigma_i$ denotes the subarc of $\gamma_-$ from $x_i$ to $y_i$, the compact region $R_i$ bounded by $I_i$, $\alpha_i$, $\sigma^-_i$ and $h(\sigma^-_i)$ is entirely contained in $W$. The set $D_i = \bigcup_{n\in \Z} g^n(R_i)$ is thus a strip bounded by the translation line $\Sigma_i = \bigcup_{n\in \Z} g^n(\alpha_i)$ and by $L_i$. See Figure \ref{fig:Gk}.
We may assume that $\alpha_1$ is disjoint from $\alpha_2$, and by our construction this implies $\Sigma_1\cap \Sigma_2=\emptyset$.

\begin{figure}[ht!]
\includegraphics[height=6cm]{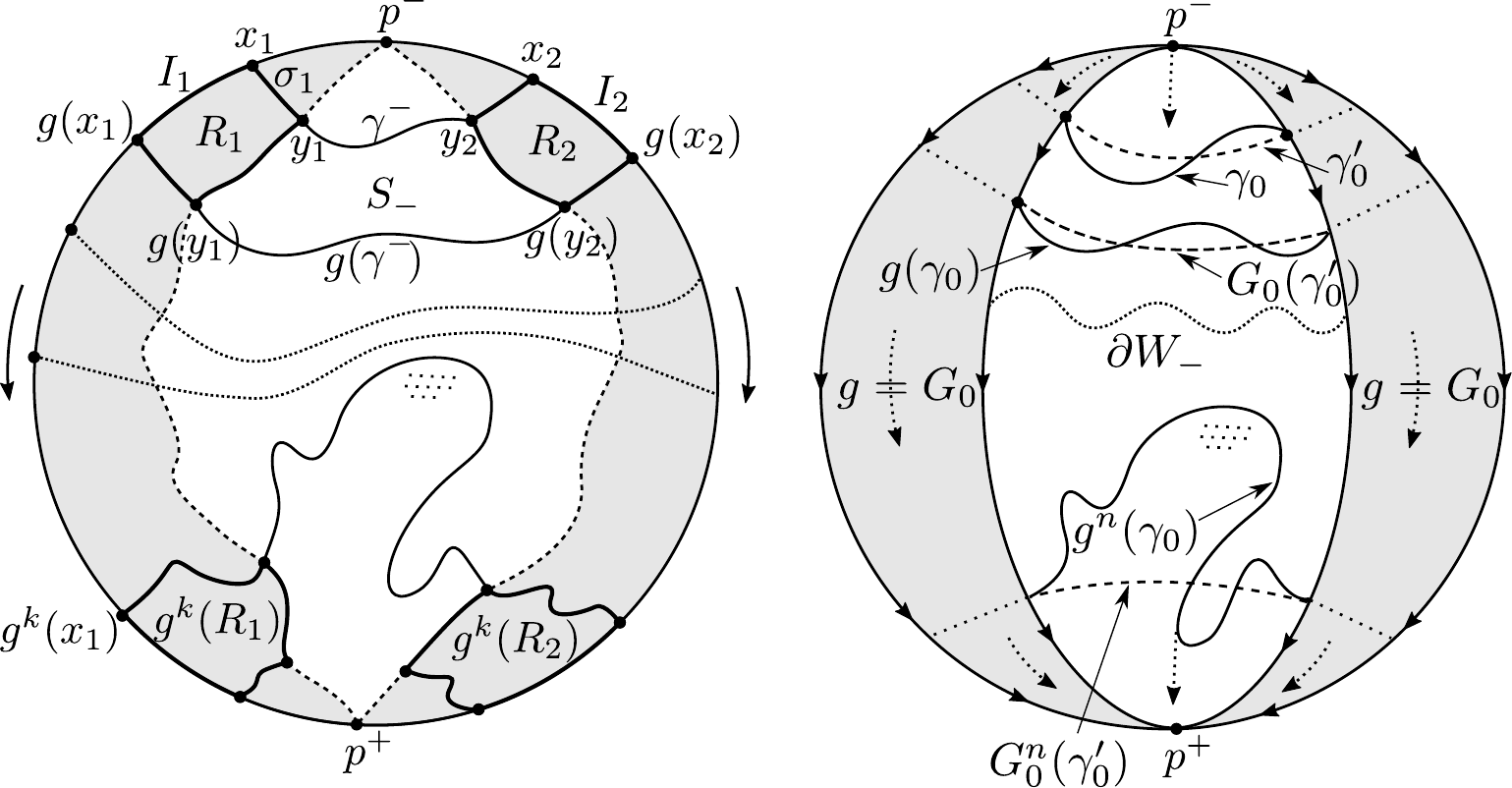}
\caption{Construction in the proof of Lemma \ref{lem:model}}
\label{fig:Gk}
\end{figure}

The previous construction may be repeated identically with the map $G_0$, obtaining sets $R_1'$, $R_2'$, $D_1'$ and $D_2'$ analogous to the previous ones. This allows us to define a map $h_0\colon D_1\cup D_2\to D_1'\cup D_2'$ which conjugates $g|_{D_1\cup D_2}$ to $G_0|_{D_1'\cup D_2'}$. Indeed, one first defines $h_0\colon \bd R_i\to \bd R_i'$ such that $h_0$ maps $\sigma_i^-$, $I_i$ and $\alpha_i$ to the corresponding arcs $\sigma_i'^-$, $I_i'$ and $\alpha_i'$ arbitrarily (but with the adequate orientation), and $h_0|_{g(\sigma_i^-)} = G_0h_0|_{\sigma_i^-}$. By Sch\"onflies' theorem, $h_0$ can be extended to a homeomorphism $R_i\to R_i'$. Then we may extend $h_0$ to $D_i$, for any $n\in \Z$, using the equation $h=G_0^nh_0g^{-n}$. The initial choice of $h_0$ guarantees that this is again a homeomorphism from $D_i$ to $D_i'$ (and since $D_1$ and $D_2$ are disjoint, may define $h_0$ in $D_1\cup D_2$ in the obvious way.

The map $h_0$ can be extended arbitrarily to a homeomorphism from $\ol\D$ to $\ol\D$ fixing $p^+$ and $p^-$. 
Thus, the map $g_0 = h_0gh_0^{-1}$ is topologically conjugate to $g$ and coincides with $G_0$ in $D_1'\cup D_2'$.
From now on we will assume that $g = G_0$ in $D_1\cup D_2$, since we can replace $g$ with $g_0$.

We now claim that given any neighborhood $W_-$ of $p^-$ in $\ol\D$ there exists a homeomorphism $h_1\colon \ol \D\to \ol \D$ and a neighborhood $W_-'\subset W_-$ of $p^-$ such that $h_1$ is the identity in $D_1\cup D_2 \cup (\ol\D\sm W_-)$, and $h_1gh_1^{-1}$ coincides with $G_0$ in $D_1\cup D_2 \cup W_-'$. In other words, we may find a homeomorphism which is topologically conjugate to $g$ and which coincides with $g$ outside a small neighborhood of $p_-$ and with $G_0$ inside a smaller neighborhood of $p_-$ and in $D_1\cup D_2$).

To prove our claim, note that if $U = \ol{\D}\sm (D_1\cup D_2\cup\{p^-, p^+\})$, then $\ol{U}$ is a closed topological disk such that $g = G_0$ in $\bd \ol{U}$ and its boundary has exactly two fixed points $p_-, p_+$, one attracting and one repelling in $\ol U$. As before we may choose a neighborhood $V_0\subset W_-$ of $p_-$ in $\ol{U}$, contained in the basin of $p_-$, such that its boundary in $U$ is an arc $\gamma_0$ joining a point of $\Sigma_1$ to a point of $\Sigma_2$ and the closure of $g^{-1}(V_0)$ is contained in $V_0$. If $S$ denotes the closure of $g(V_0)\sm V_0$, then $S$ is bounded by $\gamma_0$, $g(\gamma_0)$, a subarc $\sigma_1'$ of $\Sigma_1$ and a subarc $\sigma_2'$ of $\Sigma_2$. We may then choose an arc $\gamma_0'$ in $U$ joining the two endpoints of $\gamma_0$ such that $G_0(\gamma_0')$ is disjoint from $\gamma_0'$, and define $h_1$ on $\bd S$ so that it is the identity on $\sigma_1'\cup\sigma_2'$, $h_1(\gamma_0) = \gamma_0'$ and $h_1|_{g(\gamma_0)} = G_0h_1|_{\gamma_0}$. By Schonflies' theorem we may extend $h_1$ to a homeomorphism mapping $S$ to the region $S'$ in $\ol{U}$ bounded by $\gamma_0'$, $G_0(\gamma_0')$, $\sigma_1'$ and $\sigma_2'$. Finally we may extend $h_1$ to $V_0$ by the equation $h_1 = G_0^{-n}h_1g^n$ for $n\geq 0$. Note that the map thus defined maps $V_0$ to a subset $V_0'$ of $\ol{U}$ which is a neighborhood of $p_-$ in $G_0$ and whose boundary in $U$ is the arc $\gamma_0'$ (minus its endpoints). We also define $h_1$ as the identity in $D_1\cup D_2$, so we have $h_1$ defined in $V_0\cup D_1\cup D_2$  If $V_0$ was chosen small enough, we have that the closures of $g(V_0)$ and $G_0(V_0')$ are contained in $W_-$, and we may then extend $h_1$ to $\ol{\D}$ in such a way that $h_1$ is the identity outside $W_-$ (this can be done using Schonflies' theorem again).

Note that $g_1 = h_1gh_1^{-1}$ coincides with $G_0$ in $V_0\cup D_1\cup D_2$, which is a neighborhood of $p_-$ in $\ol{\D}$. Thus we may choose a neighborhood $W_-'\subset W_-$ of $p_-$ in $\ol{\D}$ such that $g_1 = G_0$ in $W_-'$. Moreover, $g_1 = g$ outside $W_-$, so the required properties hold.

An analogous argument can be done with $g_1$ using the attracting fixed point $p_+$ instead of $p_-$, to obtain a homeomorphism $h_2\colon \ol{\D}\to \ol{\D}$ such that $g_2 = h_2g_1h_2^{-1}$ coincides with $g_1$ outside a small neighborhood of $p_+$ and with $G_0$ in a smaller neighborhood of $p_+$ and in $D_1\cup D_2$. In particular, $g_2$ coincides with $G_0$ in a neighborhood of $\SS^1$ in $\ol{\D}$ and is topologically conjugate to $g$ as we wanted.

\end{proof}

The following is a direct consequence of the previous lemma and Theorem \ref{th:matsu-naka}:
\begin{lemma}\label{lem:flower} Suppose that $U$ is an open invariant topological disk of index $-k \leq 0$ and $f|_{\bd U}$ has no fixed points. Then there exists a map $f'$ which coincides with $f$ outside a compact subset of $U$, such that the map induced by $f'$ on the prime ends compactification $\cPE(U)$ is topologically conjugate to $G_k$. In particular, $f$ is locally conjugate to $G_k$ in a neighborhood of the circle of prime ends, and there exist $2k$ pairwise disjoint translation lines $\Gamma_0, \Gamma_1, \dots, \Gamma_{2k-1}$ embedded in $U$ such that:
\begin{itemize}
\item The $\omega$-limit of $\Gamma_i$ in $\cPE(U)$ is an attracting prime end $a_i$, the $\alpha$-limit is a repelling prime end, and they are consecutive in the cyclic ordering of $\bdPE(U)$; 
\item For each $i$, $\Gamma_i$ and $\Gamma_{i+1}$ have the same $\omega$-limit in $\bdPE(U)$ if $i$ is even and the same $\alpha$-limit if $i$ is odd (where the subindices are modulo $2k$);
\item $\bigcup_{i=0}^{2k-1} \Gamma_i$ bounds an open topological disk $U_0\subset U$ such that each component of $U\sm \ol{U_0}$ is foliated by translation lines (joining the same two prime ends);
\end{itemize}
\end{lemma}

\begin{figure}[ht!]
\includegraphics[height=5cm]{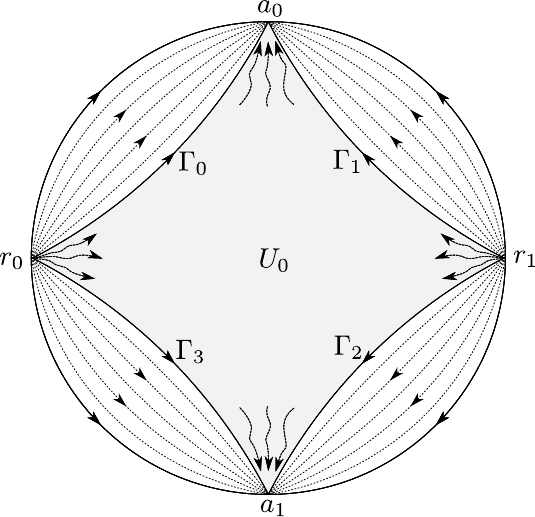}
\caption{The situation described in Lemma \ref{lem:flower}}
\label{fig:flower}
\end{figure}

\begin{remark}\label{rem:stable-branch-exist}
Note from Lemma \ref{lem:flower} that every attracting/repelling fixed prime end has some stable/unstable branch (in fact, the lines $\Gamma_i$ in Lemma \ref{lem:flower} are at the same time stable and unstable branches of $U$).
\end{remark}

The following corollary of Lemma \ref{lem:4branch} will be useful in the proof of Theorem \ref{th:pb} (but it will turn out to be vacuously true as a consequence of the theorem).
\begin{corollary}\label{coro:4branch-disk} If $U$ is an $f$-invariant open topological disk without fixed points in its boundary, with non-positive index, and with the property that $\omega(\Gamma^u) = \bd U$ for every unstable branch $\Gamma^u$ of $U$ and $\alpha(\Gamma^s) = \bd U$ for every stable branch $\Gamma^s$ of $U$, then the index of $U$ is $0$.
\end{corollary}
\begin{proof}
Suppose on the contrary that the index of $U$ is negative. Then by Lemma \ref{lem:flower}, modifying $f$ in a compact subset of $U$ we may assume that the map induced by $f$ in the prime ends compactification $\cPE(U)$ is topologically conjugate to $G_k$ for some $k\geq 1$. In particular there is a fixed point $p$ in $U$ with at least two stable and two unstable branches which converge in $\cPE(U)$ different repelling/attracting prime ends alternating in the cyclic ordering. These branches are also stable/unstable branches of $U$, and so by our hypothesis they accumulate in all of $\bd U$, which has no fixed points. This contradicts Corollary \ref{coro:4branch}.
\end{proof}

Finally, we state a special case of Lemma \ref{lem:model} for the case where there are no fixed points.

\begin{lemma}\label{lem:nofix-trans} Suppose that $g\colon \ol\D\to \ol\D$ is an orientation-preserving homeomorphism with exactly two fixed points $p^+$ and $p^-$, both in $\bd \ol \D$, such that $p^+$ is attracting and $p^-$ is repelling. Then $g$ is topologically conjugate to $G_0$.
\end{lemma}
\begin{proof}
As in the proof of Lemma \ref{lem:model}, we may choose disjoint neighborhoods $V^\pm$ of $p^\pm$ contained in the basin of $p^\pm$ and bounded by simple arcs $\gamma^\pm$ joining two points of $\bd \D$, such that $g^{\pm 1}(\ol{V}^{\pm})\subset V^\pm$. 

Let $K^\pm = \bigcap_{n\in \Z} g^n(\ol{\D}\sm V^\pm)$. Note that $K^\pm$ is a decreasing intersection of compact connected sets, hence compact and connected, and $K^\pm\cap \bd\ol{\D} = \{p^\mp\}$. We claim that $K^- = \{p^+\}$. Indeed, suppose on the contrary that $K^-\neq \{p^+\}$; then $K^-$ cannot be contained in $V^+$ (since $K$ is invariant and $V^+$ is in the basin of $p^+$), and since it is compact and invariant $K^-$ is not contained in $\bigcup_{n\in \Z} g^n(V^+)$. This implies that $K^-$ intersects $K^+$, and since $K^+$ and $K^-$ are both compact and invariant and $K^\pm\cap \bd \D = \{p^\mp\}$, it follows that $K=K^+\cap K^-$ is compact, invariant and contained in $\D$. Since $g$ has no fixed points in $\D$, this contradicts Corollary \ref{coro:brouwer}. Thus $K^- = \{p^+\}$. This implies $\bigcup_{n\in \Z} g^n(V^-) = \ol{\D}\sm \{p^+\}$, so that the region $D$ bounded by $\gamma^-$ and $g(\gamma^-)$ satisfies $\bigcup_{n\in \Z} g^n(D) = \ol{\D}\sm\{p^-, p^+\}$. Using $D$ as a fundamental domain, one may construct a topological conjugation between $g$ and $G_0$ as in the end of the proof of Lemma \ref{lem:model}.
\end{proof}

\subsection{Rotational attractor lemma}
Most of the difficulty in the proof of Theorem \ref{th:pb} is related to the possibility of the $\omega$-limit of a translation line $\Gamma$ to be topologically complicated. However, when $\Gamma$ is embedded, the very particular case where the $\omega$-limit is a circle is easier, since one may essentially replicate the proof used for flows. 

\begin{lemma}[Rotational Attractor Lemma]\label{lem:rotational} Let $\Gamma$ be a positively embedded translation line such that $\omega(\Gamma)$ has more than one point and $\rho_e(f,\Fomega(\Gamma))$ is nonzero. Then $\Fomega(\Gamma)$ is a rotational attractor. Similarly if $\Gamma$ is negatively embedded and $\rho_e(f, \Falpha(\Gamma))$ is nonzero, then $\Falpha(\Gamma)$ is a rotational repellor. 
\end{lemma}
\begin{proof}
Suppose $\Gamma$ is a positively  embedded translation line, and let $U = \SS^2\sm \Fomega(\Gamma)$. The map $f|_U$ extends to a homeomorphism $\hat{f}$ of the closed disk $\cPE(U)\simeq \ol{\D}$ such that the rotation number of $\hat f|_{\bdPE(U)}$ is nonzero. This implies that $f$ has some fixed point $p_0\in U$. Let $H = \{(x,y)\in \R^2: y\leq 0\}$, and let $\pi \colon H\to \bdPE(U)\sm \{p_0\}\simeq \D\sm \{(0,0)\}$ be the universal covering map such that $T:(x,y)\mapsto (x+1,y)$ is a generator of the group of deck transformations. Let $\Sigma\subset H$ be a lift of $\Gamma$, and choose a lift $F\colon H\to H$ of $f$ such that $F(\Sigma)\subset \Sigma$, so $\Sigma$ is a translation line for $H$. 
Since $F$ has no fixed points in the line $\bd H = \R\times \{0\}$, the first coordinate $(F(x,0)-(x,0))_1$ is different from $0$ for all $x\in \R$. We may assume without loss of generality that $(F(x,0)-(x,0))_1 > 0$ for all $x\in \R$. Since the latter map is periodic (because $F$ commutes with $T$) there exists $c>0$ such that $(F(x,0)-(x,0))_1 > c$ for all $x\in \R$. Moreover, if $\epsilon>0$ is chosen small enough, we have $(F(z)-z)_1>c$ for all $z\in \R\times [-\epsilon,0]$. Let $\eta\colon \R\to H$ be a parametrization of $\Sigma$, and denote $\Sigma_t = \eta([t,\infty))$. Note that for all $\delta>0$ there exists $t$ such that $\Sigma_t\subset \R\times [-\delta,0)$. Let $t_0$ be such that $\Sigma_{t_0}\subset \R\times (-\epsilon,0)$, and let $\sigma = \Sigma_{t_0}\sm F(\Sigma_{t_0})$, so $\Sigma_{t_0} = \bigcup_{n\geq 0} F^{n}(\sigma)$. It follows that $\Sigma$ is a closed subset of $H$. 
In fact, for each $M>0$ there exists $m$ such that $F^k(\sigma)\subset [M,\infty)\times (-\epsilon,0)$ for all $k\geq m$, and this implies that $\Sigma_t\subset [M,\infty)\times (-\epsilon,0)$ if $t$ is large enough.

Fix $M>0$ such that $\ol{\sigma} \cap [M,\infty)\times (-\epsilon,0)=\emptyset$, and let $t_1>t_0$ be such that $\Sigma_{t_1}\subset [M,\infty)\times (-\epsilon,0)$. Let $(a, b_0)$ be the initial point of $\Sigma_{t_1}$, and let $z_0=(a,b)$ be the last intersection point of $\Sigma_{t_0}$ with the segment $\{a\}\times (-\epsilon, 0)$. Let $t_2$ be such that $z_0$ is the initial point of $\Sigma_{t_2}$. Note that $z_0\notin \ol{\sigma}$, so $\Sigma_{t_2}$ is disjoint from $\sigma$. Denote by $\alpha$ the line segment $\{a\}\times (b,0)$. Then $\alpha\cup \Sigma_{t_2}\cup [a,\infty)\times \{0\}$ is a simple arc bounding an open simply connected set $D$ in $H$ (which is the connected component of $\inter H \sm (\alpha \cup \Sigma_{t_2})$ bounded to the left). Since $F(\alpha)\in (a+c,\infty)\times (-\epsilon,0)$, it follows that $F(\alpha)\cap \alpha=\emptyset$. Moreover, $F(\alpha)$ intersects $D$ (because one of its endpoints is contained in $(a+c,\infty)\times \{0\}$).  See Figure \ref{fig:ral}.
\begin{figure}[ht]
\includegraphics[width=\linewidth]{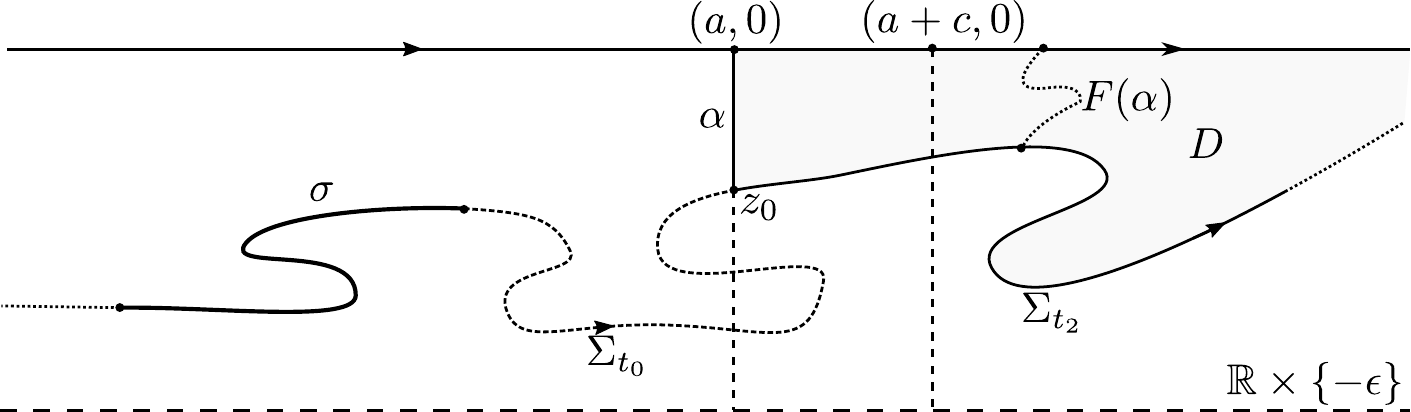}
\caption{Proof of Lemma \ref{lem:rotational}}
\label{fig:ral}
\end{figure}

Note that $\sigma$ is disjoint from $D$. Indeed, $\bd D \subset [M,\infty)\times (-\epsilon,0)$ and $D$ is bounded to the left, while $\sigma \subset (-\infty, M)\times (-\epsilon,0)$. 
We now claim that $F(D)\cap \bd D=\emptyset$. If $F(D)\cap \bd D\neq \emptyset$ then $F(D)$ intersects either $\Sigma_{t_2}$ or $\alpha$. Suppose first that $F(D)\cap \alpha\neq \emptyset$. Note that $\ol{F(D)}$ is disjoint from $(-\infty,a]\times \{0\}$, so $F(D)$ does not contain $\alpha$ entirely. Thus $\bd F(D)\cap \alpha\neq \emptyset$. But $F(\alpha)\cap \alpha=\emptyset$, so $F(\Sigma_{t_2})$ intersects $\alpha$. This is impossible since $F(\Sigma_{t_2})\subset \Sigma_{t_2}\subset \Sigma_{t_0}$ which is disjoint from $\alpha$ by construction.
Now assume that $F(D)\cap \Sigma_{t_2}\neq \emptyset$. This implies that $D\cap F^{-1}(\Sigma_{t_2})\neq \emptyset$. Since $D$ is disjoint from $\Sigma_{t_2}$, it follows that $D$ intersects $F^{-1}(\Sigma_{t_2})\sm \Sigma_{t_2}$. The latter set is contained in $\Sigma_{t_0}\sm \Sigma_{t_2}$, which is an arc containing $\sigma$. Since $\sigma\cap D=\emptyset$, it follows that $\bd D$ intersects $\Sigma_{t_0}\sm \Sigma_{t_2}$. But since $\alpha$ is disjoint from $\Sigma_{t_0}$, this is impossible. 

Thus, $F(D)\cap \bd D=\emptyset$, and since $F(D)$ intersects $D$, we conclude that $F(D)\subset D$. To finish the proof, let us show that for each $\delta>0$ there exists $n_0$ such that $F^n(D) \subset \R\times (-\delta,0]$ whenever $n>n_0$. To show this, we may assume that $\delta<\epsilon$. Note that $F^k(D)\subset [a+kc, \infty)\times (-\epsilon,0)$. Let $t_3>0$ be such that $\Sigma_{t_3}\subset \R\times (-\delta,0)$. Let $n_0$ be such that $\Sigma_{t_0}\sm \Sigma_{t_3}$ is disjoint from $[a+n_0c,\infty)\times (-\epsilon,0)$. If $n>n_0$ and $(x,y)\in F^n(D)$, then $-\epsilon<y<0$ and $x\geq a+n_0c$. The line segment joining $(x,-\epsilon)$ to $(x,y)$ must have a first intersection point $(x,y')$ with $\bd D$, and clearly $(x,y')\in \Sigma_{t_2}$. Moreover, since $\Sigma_{t_0}\sm \Sigma_{t_3}$ is disjoint from $[a+n_0c,\infty)\times (-\epsilon,0)$ and $x\geq a+n_0c$, it follows that $(x,y')\in \Sigma_{t_3}$. Thus $(x,y')\in \R\times (-\delta, 0]$, and since $y'\geq y$ we also have $(x,y)\in \R\times (-\delta,0]$, as claimed.

Finally, note that $W=\pi(D)\cup \Fomega(\Gamma)$ is a neighborhood of $\Fomega(\Gamma)$, and from the previous paragraph we have that $f(W) \subset W$ and for each $\delta>0$ there is $n_0$ such that $f^n(W)$ is contained in the $\delta$-neighborhood of $\Fomega(\Gamma)$ whenever $n>n_0$. This implies that $\Fomega(\Gamma) = \bigcap_{n\geq 0} f^n(\ol{W})$, hence we conclude from Lemma \ref{lem:shub} that $\Fomega(\Gamma)$ is an attractor, as we wanted.
\end{proof}

The next fact is essentially contained in the proof of Lemma \ref{lem:rotational} (combining it with Lemma \ref{lem:ac-PE-bd}).
\begin{lemma}\label{lem:rotational-spiral}
If $\Gamma$ is a translation ray in the basin of a rotational attractor $A$ and $\Gamma\cap A = \emptyset$ then $\Fomega(\Gamma) = A$ and $\Gamma$ spirals towards $A$.
\end{lemma}

\section{A Poincar\'e-Bendixson theorem for translation lines}

This section is devoted to the proof of Theorem \ref{th:pb}. For the remainder of this section $f\colon \SS^2\to \SS^2$ will denote a homeomorphism homotopic to the identity.
We recall that if a translation line is disjoint from a rotational attractor $A$ but contained in its basin, then it spirals towards $A$ (see Lemma \ref{lem:rotational-spiral}).

\subsection{A result in the fixed point free case in the annulus}

Denote by $\A=\T^1\times \R$ the open annulus and by $\pi\colon \R^2\to \A$ its universal covering map such that $T\colon (x,y)\mapsto (x+1,y)$ generates its group of deck transformations.

\begin{lemma} \label{lem:oneside} Let $\Sigma$ be an arc that projects injectively into $\A$, and let $B\subset \R^2$ be a closed topological disk also projecting injectively into $\A$ (so $B\cap T^iB = \emptyset$ and $\Sigma \cap T^i\Sigma=\emptyset$ for all $i\neq 0$). If $\Sigma$ intersects both $T^iB$ and $T^jB$ for $i,j\in \Z$, then $\Sigma$ also intersects $T^kB$ whenever $i\leq k\leq j$.
\end{lemma}
\begin{proof}
We may assume that $\Sigma$ is compact, replacing it by a sufficiently large subarc.
Suppose for a contradiction that there exists $k$ with $i<k<j$ such that $\Sigma$ is disjoint from $T^kB$. By 
translating everything with $T^{-k}$ we may assume that $i<0<j$ and $\Sigma$ is disjoint from $B$. 

Let $\gamma$ be the subarc of $\Sigma$ joining the last intersection of $\Sigma$ with $\bigcup_{l<0} T^l B$ to the first intersection of $\Sigma$ with $\bigcup_{l>0} T^lB$. Let $x_0 \in T^aB$ and $x_1\in T^bB$ be the two endpoints of $\gamma$, so $a< 0$ and $b> 0$. Then $K = \gamma\cup T^bB$ is a compact connected set, and using the fact that $\gamma \cap T^lB = \emptyset$ for $a<l<b$ and the hypotheses on $\Sigma$ and $B$ one easily verifies that $TK\cap K=\emptyset$. By Lemma \ref{lem:free-homeo} applied to $h=T$, it follows that $T^iK\cap K=\emptyset$ for all $i\neq 0$. Since $T^{b-a}K\cap K \neq \emptyset$ and $b-a\geq 2$, we have the required contradiction.
\end{proof}

\begin{lemma}\label{lem:omegama}  Let $h\colon \A\to \A$ be a homeomorphism isotopic to the identity without fixed points and $\Gamma\subset \A$ a translation line such that $\ol{\Gamma}$ is compact. Then $\alpha(\Gamma)\cap \omega(\Gamma)=\emptyset$.
\end{lemma}

\begin{proof}
Let $\Sigma$ be a lift of $\Gamma$ to $\R^2$, and let $H\colon \R^2\to \R^2$ be a lift of $h$ such that $H(\Sigma)=\Sigma$, so $\Sigma$ is a translation line of $H$. Note that, since $h$ is isotopic to the identity, $TH=HT$. 
Fix $x_0\in \Sigma$ and let $\Sigma_0$ be the compact subarc of $\Sigma$ from $x_0$ to $H(x_0)$. Let $\Sigma^+ = \bigcup_{n>0} H^n(\Sigma_0)$ and $\Sigma^-=\bigcup_{n<0} H^n(\Sigma_0)$, so $\Sigma = \Sigma^-\cup \Sigma_0\cup \Sigma^+$. 

Suppose that $\omega(\Gamma)\cap \alpha(\Gamma)$ is nonempty. Then it is a compact invariant set, so it contains some bi-recurrent point $z$ (\ie both the forward and backwards orbit of $z$ accumulate in $z$). We may assume $z\notin \Gamma_0:=\pi(\Sigma_0)$ (replacing it by some iteration by $h$ if necessary). Let $B\in \A$ be a small enough closed topological disk containing $z$ in its interior, disjoint from $\Gamma_0$ and such that $h(B)\cap B=\emptyset$. Let $\cB\subset \R^2$ be a lift of $B$, and $\til z$ the element of $\pi^{-1}(z)$ in $\cB$. Let $$I^\pm(\cB) = \{i\in \Z : \Sigma^\pm\cap T^i\cB\neq \emptyset\}.$$

Since $\pi(\cB)$ intersects both $\omega(\Gamma)$ and  $\alpha(\Gamma)$, the two sets $I^+(\cB)$ and $I^-(\cB)$ are nonempty.
Note that Lemma \ref{lem:oneside} applied to $\Sigma^\pm$ implies that $I^\pm(\cB)$ is an interval of integers (\ie if $a,b \in I^+(\cB)$ then $c\in I^+(\cB)$ whenever $a\leq c\leq b$, and similarly for $I^-(\cB)$). 

We claim that $I^+(\cB)\cap I^-(\cB)=\emptyset$. Suppose on the contrary that there exists $i$ such that both $\Sigma^+$ and $\Sigma^-$ intersect $T^i\cB$. Let $\alpha$ be the subarc of $\Sigma$ joining the last point of $\Sigma^-\cap T^i\cB$ to the first point of $\Sigma^+\cap T^i\cB$ (in the linear order of $\Sigma$). Then $\alpha$ intersects $T^i\cB$ only at its endpoints. Since $T^i\cB$ is disjoint from $H(T^i\cB)$ and $H$ has no fixed points, Lemma \ref{lem:trans-fix} implies that $H(\alpha)\cap \alpha=\emptyset$. But this contradicts the fact that $\Sigma_0\subset \alpha$. Thus $I^-(\cB)\cap I^+(\cB)=\emptyset$.

In addition, note that Lemma \ref{lem:oneside} applied to $\Sigma$ implies that $I^+(\cB)\cup I^-(\cB)$ is also an interval of integers. Choose $i^\pm\in I^\pm(\cB)$ and assume without loss of generality that $i^-<i^+$ (the case $i^->i^+$ is analogous). Then, from the previous facts we deduce that $k := \max I^-(\cB)$ is finite, $k+1\in I^+(\cB)$, and $I^-(\cB)\subset (-\infty, k]$ while $I^+\subset [k+1, \infty)$.

If $B_0\subset B$ is a smaller closed topological disk containing $z$ in its interior and $\cB_0$ is the lift of $B_0$ in $\cB$, then again $I^\pm(\cB_0)$ is nonempty and $I^+(\cB_0)\cup I^-(\cB_0)$ is an interval of integers. Moreover, it is clear that $I^\pm(\cB_0)\subset I^\pm(\cB)$.  From these facts we deduce that $k\in I^-(\cB_0)$ and $k+1\in I^+(\cB_0)$. 
In other words, $\Sigma^-$ intersects $T^k\cB_0$ and $\Sigma^+$ intersects $T^{k+1}\cB_0$. Since $\cB_0$ can be chosen arbitrarily small, we deduce that $T^k\til z\in \ol{\Sigma}^-$ and $T^{k+1}\til z\in \ol{\Sigma}^+$.

On the other hand, recall that $z$ is a bi-recurrent point, so we can choose $n$ such that $h^n(z)\in B$, and we may choose $n$ to be positive or negative as we wish. Note that $z$ could belong to $\pi(\Sigma^+)$ or $\pi(\Sigma^-)$, but it is always the case that either $z\notin \pi(\Sigma^-)$ or $z\notin \pi(\Sigma^+)$.

Assume first that $z\notin \pi(\Sigma^-)$, so $T^k\til{z}\notin \Sigma^-$. Since $T^k\til z$ is accumulated by $\Sigma^-$, this means that $T^k\til{z}\in \alpha(\Sigma)$. Since $H(\Sigma^+)\subset \Sigma^+$, by choosing $n$ positive we may guarantee that $H^n(\ol{\Sigma}^+)\subset \ol{\Sigma}^+$. 
Note that $H^n(\til z)\in T^j\cB$ for some $j$, and since $H(\cB)\cap \cB=\emptyset$ and $H$ has no fixed points, Lemma \ref{lem:free-homeo} implies that $j\neq 0$. Moreover, since $H^n(\alpha(\Sigma))=\alpha(\Sigma)$ and $H^n(\ol{\Sigma}^+)\subset \ol{\Sigma}^+$, we deduce that $T^{k+j}\til z \in \alpha(\Sigma)$ and $T^{k+1+j}\til z\in \ol{\Sigma}^+$. This implies that $k+j\in I^-(\cB)$ and $k+j+1\in I^+(\cB)$. But this is a contradiction since $j\neq 0$, so either $k+j>k$ or $k+j+1\leq k$ (whereas $I^-(\cB)\subset (-\infty, k]$ and $I^+(\cB)\subset [k+1, \infty)$).

Similarly, if one assumes that $z\notin \pi(\Sigma^+)$, one has $T^{k+1}\til z \in \omega(\Sigma)$ while $T^k\til z\in \ol{\Sigma}^-$. Using the fact that $\omega(\Sigma)$ and $\ol{\Sigma}^-$ are forward invariant by $h^{-1}$, one may repeat the argument from the previous paragraph using $n<0$, which leads to a similar contradiction. This completes the proof.
\end{proof}

\subsection{A special case}

\begin{lemma}\label{lem:nonegindex} Suppose that $\Gamma$ is a translation line, $\omega(\Gamma)$ has no fixed points, and every invariant connected component of $\SS^2\sm \omega(\Gamma)$ has nonnegative index. Then $\Fomega(\Gamma)$ is a rotational attractor.
\end{lemma}
\begin{proof}
By Corollary \ref{coro:disk-index} we know that every invariant connected component of $\SS^2\sm \omega(\Gamma)$ has index at most $1$.
Note that there are finitely many such components containing fixed points, because $\fix(f)\cap (\SS^2\sm \omega(\Gamma))$ is compact. Moreover, since the sum of their indices is $2$, two of these components are invariant disks of index $1$, and all other components are invariant disks of index $0$. Thus using Lemma \ref{lem:model} on each component of index $0$ we may modify $f$ outside a neighborhood of $\omega(\Gamma)$, to obtain a map $h$ having exactly two fixed points, both of index $1$. We denote the two fixed points by $\infty$ and $-\infty$. Thus there are no fixed points in the annulus $\A=\SS^2\sm \{-\infty, \infty\}$. 
Let $U_\infty$ and $U_{-\infty}$ be the connected components of $\SS^2\sm \omega(\Gamma)$ containing $
\infty$ and $-\infty$, respectively. Since the prime ends rotation numbers of $f$ (hence of $g$) in $U_\infty$ and $U_{-\infty}$ are nonzero (by Lemma \ref{lem:rot-index}), by an additional modification of $h$ (which we still denote $h$) outside a neighborhood of $\omega(\Gamma)$ we may assume that there exist two invariant closed topological disks $D^+\subset U_\infty$ and $D^-\subset U_{-\infty}$ such that $\pm \infty \in D^{\pm}$ (see \cite[Proposition 5.1]{franks-lecalvez}). 

Note that the line $\Gamma$ may fail to be $h$-invariant, but since $f$ coincides with $h$ in a neighborhood of $\omega(\Gamma)$,  there exists $z\in \Gamma$ such that the ray $\Gamma_z^+$ starting at $z$ and following $\Gamma$ positively is a positive translation ray for $h$ and is disjoint from $D^+\cup D^-$. Thus $\Gamma' = \bigcup_{n\geq 0} h^{-n}(\Gamma^+_z)$ is a translation line for $h$ with the same $\omega$-limit as $\Gamma$. Furthermore, since the annulus $\SS^2\sm (D^+\cup D^-)$ is invariant, $\Gamma'\subset \SS^2\sm (D^+\cup D^-)$ and in particular $\ol{\Gamma'} \subset A:= \SS^2\sm \{\infty, -\infty\}$.
If we prove that $\Fomega(\Gamma')$ is a rotational attractor for $h$, it follows immediately that $\Fomega(\Gamma)$ is a rotational attractor and we are done.

By Lemma \ref{lem:omegama} applied to the annulus $A\simeq \A$ we know that $\omega(\Gamma')$ is disjoint from $\alpha(\Gamma')$. This implies that $\Gamma'$ is an embedded line. Moreover, $\Fomega(\Gamma')$ and $\Falpha(\Gamma')$ are two disjoint cellular continua in $\SS^2$, and so each must contain a fixed point. Since the only fixed points are $\pm \infty$, one of these sets contains $\infty$ and the other contains $-\infty$. Since each of these points has index $1$, we deduce that the invariant disk $\SS^2\sm\Fomega(\Gamma')$ has index $1$. Thus by Corollary \ref{coro:disk-index} $\rho(h, \Fomega)$ is nonzero, and by Lemma \ref{lem:rotational} we deduce that $\Fomega(\Gamma')$ a rotational attractor, as we wanted.
\end{proof}

\subsection{The embedded case}
\begin{lemma}\label{lem:embedded} The conclusion of Theorem \ref{th:pb} holds if $\Gamma$ is an embedded line.
\end{lemma}
\begin{proof}
Assume for a contradiction that $\omega(\Gamma)$ has no fixed points and $\Fomega(\Gamma)$ is not a rotational attractor.
By Lemma \ref{lem:rotational}, we have $\rho_e(f,\Fomega(\Gamma))=0$. Letting $U_0 = \SS^2\sm \Fomega(\Gamma)$, we thus have $\rho(f,U_0)=0$. By Corollary \ref{coro:disk-index}, it follows that $U_0$ has nonpositive index. Note that $\bd U_0 = \bd \Fomega(\Gamma) =\omega(\Gamma)$.

There exists an open topological disk $U$ which is maximal with respect to the following properties:
\begin{itemize}
\item $U_0\subset U$;
\item $U$ is invariant and has nonpositive index;
\item $\bd U\subset \bd U_0$.
\end{itemize}
The existence of $U$ follows from a standard Zorn's Lemma argument: if $\mc{C}$ is a chain (with respect to inclusion) of invariant disks satisfying the three conditions above, then $V' = \bigcup_{V\in \mc C} V$ is an open invariant topological disk containing $U_0$ and $\bd V'\subset \bd U_0$. Since there are no fixed points in $\bd U_0$, the set of fixed points in $V'$ is compact and therefore it is contained in some $V\in \mc C$. Since $V$ has nonpositive index, it follows that $V'$ has nonpositive index. This guarantees that we can apply Zorn's Lemma.

Let $K=\bd U$. We claim that if $\Gamma^s$ is any stable branch of $U$, then $\alpha(\Gamma^s)=K$. 
First let us show that $\rho_e(f, \Falpha(\Gamma^s))=0$. Indeed, if this is not the case then Lemma \ref{lem:rotational} implies that $\Falpha(\Gamma^s)$ is a rotational repellor. However, $\alpha(\Gamma^s)\subset K\subset \bd U_0  = \omega(\Gamma)$, which implies that the $\omega$-limit of the translation line $\Gamma$ intersects the repellor $\Falpha(\Gamma^s)$. By Lemma \ref{lem:attractor-intersect} this means that $\Gamma\cap \Falpha(\Gamma^s)\neq \emptyset$, contradicting the fact that $\Gamma$ is disjoint from $K$.
This shows that $\rho_e(f, \Falpha(\Gamma^s))=0$. But then $U'=\SS^2\sm \Falpha(\Gamma^s)$ is an open invariant topological disk of nonpositive index (due to Corollary \ref{coro:disk-index}) containing $U$, and from the maximality of $U$ we deduce that $U'=U$, so $\alpha(\Gamma^s) = \bd U' = \bd U = K$ as claimed.

Now we claim that any unstable branch $\Gamma^u$ of $U$ also satisfies $\Fomega(\Gamma^u)=K$. The proof is similar, except that to conclude $\rho_e(f, \Fomega(\Gamma^u))=0$ we use the fact that there exists an stable branch $\Gamma^s$ of $U$ such that $\Falpha(\Gamma^s)=K$ (by the previous paragraph and Remark \ref{rem:stable-branch-exist}) 
which implies that $\Fomega(\Gamma^u)$ cannot be a rotational attractor (using Lemma \ref{lem:attractor-intersect} as in the previous paragraph). 

From the previous claims and Corollary \ref{coro:4branch-disk} we conclude that the index of $U$ is $0$.
Finally, we will show that if $V$ is any invariant connected component of $\SS^2\sm K$ of nonpositive index then $V$ has the property that the $\omega$-limit of every unstable branch and the $\alpha$-limit of every stable branch of $V$ are equal to $\bd V$. Note that this will also imply that $V$ has index $0$. 

Let $\Gamma^u$ be an unstable branch of $V$. We claim that $U\subset \Fomega(\Gamma^u)$.
Note that since $U$ is disjoint from $\omega(\Gamma^u)=\bd \Fomega(\Gamma^u)\subset \bd V\subset K$, it follows that either $U\subset \Fomega(\Gamma^u)$ or $U\subset \SS^2\sm \Fomega(\Gamma^u)$. 
Assume for a contradiction that $U\subset \SS^2\sm \Fomega(\Gamma^u)$ and suppose first that $\rho_e(f,\Fomega(\Gamma^u))\neq 0$. Then Lemma \ref{lem:rotational} implies that $\Fomega(\Gamma^u)\subset \bd V$ is a rotational attractor. But since $U$ contains stable branch which accumulates in all of $K \supset \bd \Fomega(\Gamma^u)$, the $\alpha$-limit of such a stable branch intersects the attractor $\Fomega(\Gamma^u)$ which is not possible by Lemma \ref{lem:attractor-intersect}. Now suppose that $\rho_e(f, \Fomega(\Gamma^u))=0$. Then $U'=\SS^2\sm \Fomega(\Gamma^u)$ is an open invariant topological disk strictly containing $U$, and it has nonpositive index due to Corollary \ref{coro:disk-index}. This contradicts the maximality of $U$, completing the proof that $U\subset \Fomega(\Gamma^u)$.

We deduce that $K=\bd U \subset \Fomega(\Gamma^u)$ as well, which implies that $\bd V \subset \Fomega(\Gamma^u)$. Since $\Gamma^u\subset V$, this means that $\bd V = \omega(\Gamma^u)$. An analogous argument shows that $\bd V = \alpha(\Gamma^s)$ for any stable branch $\Gamma^s$ of $V$. Thus every unstable branch of $V$ has an $\omega$-limit equal to $\bd V$ and similarly every stable branch has $\alpha$-limit $\bd V$, as we wanted to show.

Summarizing, we have found an invariant disk $U$ of index $0$, and a continuum $K=\bd U$ such that every invariant component of $\SS^2\sm K$ has index $0$ or $1$ (since the components of positive index must have index $1$ by Corollary \ref{coro:disk-index}). This implies that there are exactly two components of index $1$. Furthermore, we have shown that every stable or unstable branch of $U$ accumulates in all of $K$. In particular, an unstable branch $\Sigma'$ of $U$ is an embedded translation line such that $\omega(\Gamma')=\bd U = K$, which has no fixed points. Clearly $\Fomega(\Gamma') = \SS^2\sm U$ is not a rotational attractor (otherwise $U$ would have index $1$), which contradicts Lemma \ref{lem:nonegindex}. This completes the proof of the lemma.
\end{proof}

\subsection{The general case} 
To complete the proof of Theorem \ref{th:pb}, we need to consider the case where $\Gamma$ is not necessarily embedded. 
Assume for a contradiction that $\omega(\Gamma)$ is not a rotational attractor and it contains no fixed points. In view of Lemma \ref{lem:nonegindex}, we may also assume that there exists some invariant connected component $U_0$ of $\SS^2\sm \omega(\Gamma)$ of negative index. 

First note that by Lemma \ref{lem:flower} we may choose an embedded translation line $\Gamma'\subset U_0$ such that, in the prime ends compactification of $U_0$, the $\omega$-limit of $\Gamma'$ is an attracting fixed prime end and its $\alpha$-limit is a repelling fixed prime end. Note that $\alpha(\Gamma')\subset \bd U_0 \subset \omega(\Gamma)$, which contains no fixed points of $f$. By Lemma \ref{lem:embedded} applied to $\Gamma'$ (using $f^{-1}$ instead of $f$) we deduce that $K:=\Falpha(\Gamma')$ is a rotational repellor.

By Lemma \ref{lem:attractor-intersect},  $\Gamma$ intersects $K$. Clearly $\omega(\Gamma)$ is not contained in $K$, because  $\omega(\Gamma')$ is disjoint from $K$ and contained in $\bd U_0\subset \omega(\Gamma)$. Thus, for any $z\in \Gamma$, the subarc $\Gamma^+_z$ starting at $z$ and moving forward through $\Gamma$ contains points in $K$ and points outside of $K$. Fix $z'\in \Gamma'$ and let $\Gamma'^-_{z'}$ be the ray starting at $z'$ and following $\Gamma'$ negatively. Then $\Gamma'^-_{z'}$ spirals towards $K$, so by Lemma \ref{lem:spiral} we have that $\Gamma^+_z$ intersects $\Gamma'^-_{z'}$. Since $\Gamma^+_z$ is forward invariant, it follows that $\Gamma^+_z$ intersects the fundamental domain $\Gamma'_0$ defined between $z'$ and $f^{-1}(z')$ in $\Gamma'^-_{z'}$.  This holds for all $z\in \Gamma$, so we conclude that $\omega(\Gamma)$ intersects $\Gamma'_0$. But this is a contradiction, since $\Gamma'_0\subset U_0$ which is disjoint from $\omega(\Gamma)$. This completes the proof of the theorem. \qed

\section{Prime ends and fixed points: proof of Theorem \ref{th:disks}}

\subsection{Spiraling lines and principal sets}
We begin with a general topological result. As before, $f\colon \SS^2\to \SS^2$ will always denote an orientation-preserving homeomorphism.

\begin{lemma}\label{lem:im} If $U$ is an open topological disk and $\Gamma\subset U$ is an embedded line which converges towards a prime end $\pe p$ in $\cPE(U)$ and spirals towards $\omega(\Gamma)$ in $\SS^2$. Then $\omega(\Gamma)$ coincides with the principal set of $\pe p$.
\end{lemma}
\begin{proof}
Let $U' = \SS^2\sm \Fomega(\Gamma)$, so that $\bd U' = \omega(\Gamma)$. 
We first remark that $U\subsetneq U'$, since otherwise $\Gamma$ would not spiral toward $\omega(\Gamma)$ (recall the definition from Section \ref{sec:spiral}). 

By Proposition \ref{pro:pri-im}, the principal set $\priPE(\pe p)$ is contained in $\omega(\Gamma)$. Moreover, by the same proposition we may assume there exists a ray $\Gamma'$ in $U$ converging towards $\pe p$ in $\cPE(U)$ such that $\omega(\Gamma') = \priPE(\pe p) \subset \omega(\Gamma)$.  To complete the proof it suffices to show that $\omega(\Gamma) = \omega(\Gamma')$.

We identify $\cPE(U')$ with $\ol{\D}$ in a way that the origin does not belong to $U$ (which is possible since $U\subsetneq U'$). Let $A = \ol{\D}\sm \{(0,0)\}$ and let $\til{A}$ denote its universal covering, which we may identify with $\{(x,y)\in \R^2 : y\leq 0\}$, with the translation $T\colon(x,y)\mapsto  (x+1, y)$ generating the group of deck transformations. 
Let $\til \Gamma$ be a lift of $\Gamma$ to $\til{A}$, and fix $y_0<0$ such that $\til \Gamma$ intersects the line  $L_0 = \{(x,y_0):x\in \R\}$. Note that if $\gamma\colon [0,\infty)\to \til A$ is a parametrization of $\til \Gamma$, denoting by $\pr_i$ the projection onto the $i$-th coordinate we have $\pr_2(\gamma(t))\to 0$ and $\pr_1(\gamma(t))\to \infty$ or $-\infty$ as $t\to \infty$ (because $\Gamma$ spirals towards $\SS^1$ in $\ol{\D}$).

Thus there exists a last intersection of $\til{\Gamma}$ with $L_0$, \ie there exists $t_0$ such that the image $\til \Gamma_{t_0}$ of $\gamma|_{[t_0, \infty)}$ intersects $L_0$ only at its initial point $\gamma(t_0)$. Consider the region $D$ bounded by $\til \Gamma_{t_0}$ and the half-line $\{(x, y_0) : x \geq \pr_1(\gamma(t_0))\}$, which is a topological disk whose projection $\pr_1(D)$ is bounded from below (see Figure \ref{fig:lem-im}). The set $D$ has the property that for every $z\in \til A$ with $y_0<\pr_1(z)<0$ there exists $i_0$ such that $T^i z\in D$ for all $i\geq i_0$.

\begin{figure}[ht!]
\includegraphics[height=2.5cm]{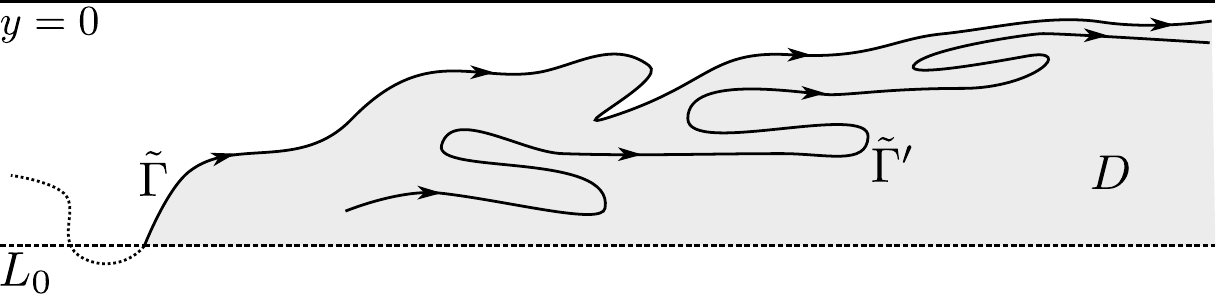}
\caption{Proof of Lemma \ref{lem:im}}
\label{fig:lem-im}
\end{figure}

We claim that $\Gamma'$ also spirals towards $\SS^1$ in $\ol{\D}$.
Recall that $\Gamma'$ was chosen as a ray in $U$ such that $\Fomega(\Gamma')=\priPE(\pe p)\subset \Fomega(\Gamma)$ in $\SS^2$. This implies that in $\ol{\D}$, the $\omega$-limit of $\Gamma'$ is contained in the boundary circle. In particular, removing an initial segment of $\Gamma'$ if necessary, we may assume that any lift $\Gamma'$ is disjoint from the projection of $L_0$ to $\ol{\D}$ (which is a circle). Let $\til \Gamma'$ be a lift of $\Gamma'$ to $\til A$, so by our assumption $\pr_2(\til\Gamma')\subset (y_0, 0)$. We may choose the lift so that $\til \Gamma'$ intersects $D$ (replacing $\til \Gamma'$ with $T^i(\til \Gamma')$ for a large enough $i$). In addition, we may assume that $\til \Gamma'$ is disjoint from $\til \Gamma_{t_0}$, since $T^i\til\Gamma'$ intersects $\til \Gamma_{t_0}$ for at most one value of $i$ (as both sets project into the simply connected subset $U$ of $\ol{\D}\sm \{(0,0)\}$).

Thus $\til \Gamma'$ intersects $D$ but not $\bd D$, which implies that $\til \Gamma'\subset D$. Moreover, if $\gamma'$ is a parametrization of $\til \Gamma'$, one has $\pr_2(\gamma'(t))\to 0$ as $t\to\infty$. This, together with the fact that $\pr_1(\gamma(t))\to\infty$ as $t\to\infty$, easily implies that $\pr_1(\gamma'(t))\to \infty$. Thus $\Gamma'$ spirals towards $\SS^1$ in $\ol{\D}$, which means that, seen in $\SS^2$, the ray $\Gamma'$ spirals towards $\bd U'$, and in particular $\omega(\Gamma') = \bd U' = \omega(\Gamma)$ as we wanted to show.
\end{proof}

\subsection{Attractors and repellors of fixed prime ends}

The previous lemma and Theorem \ref{th:pb} lead to the following:

\begin{theorem}\label{th:prime-end} Suppose $U$ is an $f$-invariant open topological disk and $\pe p$ is a fixed prime end. If the impression of $\pe p$ has no fixed points, then:
\begin{itemize}
\item[(1)] $\pe p$ is an attracting (or repelling) prime end;
\item[(2)] The principal set of $\pe p$ is the boundary of a rotational attractor (or repellor) $K$ disjoint from $U$;
\item[(3)] Every translation ray (or negative translation ray) $\Gamma$  in $U$ converging towards $\pe p$ in $\cPE(U)$ spirals towards $K$ in $\SS^2$, and there exist such rays;
\item[(4)] The interior of $K$ has a unique invariant connected component, of fixed point index $1$.
\end{itemize}
\end{theorem}

\begin{proof} By Lemma \ref{lem:C-L} we have that $\pe p$ is either attracting or repelling in $\cPE(p)$, so (1) holds. Assume that $\pe p$ is attracting (the repelling case is similar). Then there exists an embedded translation ray $\Gamma$ in $U$ whose $\omega$-limit in $\cPE(p)$ is $\pe p$. By Theorem \ref{th:pb}, either $\omega(\Gamma)$ contains a fixed point or $\Fomega(\Gamma)$ is a rotational attractor. Since $\omega(\Gamma)$ is contained in the impression of $\pe p$ (see Section \ref{sec:prime-ends}), by our assumption it has no fixed points, so $K= \Fomega(\Gamma)$ must be a rotational attractor and $\Gamma$ spirals towards $K$. Since $\omega(\Gamma)\subset \bd U$ and $\Gamma\subset U$, the definition of filled $\omega$-limit implies that $K$ is disjoint from $U$. From the previous lemma, we have that $\omega(\Gamma) = \bd K$ is the principal set of $\pe p$, proving (2).
Since the same argument applies to any translation ray $\Gamma\subset U$ converging towards $\pe p$ in $\cPE(U)$, item (3) follows.

For the last item, first note that since $K$ is a rotational attractor, $\rho(f, \SS^2\sm K)\neq 0$ and so by Lemma \ref{lem:rot-index} the fixed point index of the disk $\SS^2\sm K$ is $1$. Since $\bd K\subset \bd U$ has no fixed points, the sum of the indices of all invariant connected components of $\SS^2\sm \bd K$ must be equal to $2$. One of these components is $\SS^2\sm K$, which has index $1$, and all other components are components of the interior of $K$. Therefore the sum of the indices of invariant connected components of $\inter(K)$ must be equal to $1$.
Hence to prove (4) it suffices to show that all invariant components of $\inter(K)$ have index equal to $1$.

Let $U'$ be any invariant connected component of the interior of $K$, and assume for a contraction that its fixed point index is not $1$.  Since $\bd U'$ has no fixed points, its index is at most $1$ by Corollary \ref{coro:disk-index}, so it must be nonpositive. By Lemma \ref{lem:rot-index} one has $\rho(f, U') = 0$, and by Theorem \ref{th:matsu-naka} there exists at least one attracting and one repelling fixed prime end of $U'$. By the previous items applied to $U'$, there is at least one rotational attractor $A$ and a rotational repellor $R$, both disjoint from $U'$, such that $\bd A$ and $\bd R$ are the principal sets of a fixed attracting and repelling prime end of $U'$, respectively. Moreover, as in (2) there exists a translation ray $\Gamma$ in $U$ spiraling towards $K = \Fomega(\Gamma)$, and there exists a translation ray $\Gamma'$ in $U'$ spiraling towards $A = \Fomega(\Gamma')$. 
Since $\bd R\subset \bd K = \omega(\Gamma)$, by Lemma \ref{lem:attractor-intersect} we have that $\Gamma$ intersects the repellor $R$. Since $\Gamma$ is disjoint from $\bd R\subset \bd U$, it follows that $\Gamma\subset R$, thus $\bd K = \omega(\Gamma)\subset R$. But then $\omega(\Gamma') \subset \bd U'\subset \bd K \subset R$, so again by Lemma \ref{lem:attractor-intersect} we have conclude that $\Gamma'\cap R\neq \emptyset$. This is a contradiction, since $\Gamma'\subset U'$ which is disjoint from $R$. This contradiction shows that the fixed point index of $U'$ is $1$ for every invariant connected component of the interior of $K$, as we wanted.
\end{proof}

We note that item (4) only uses the fact that $K$ is a rotational attractor or repellor and has no fixed points. Thus we have the following:
\begin{corollary}\label{coro:rotational-index} If $K$ is a rotational attractor or repellor of $f$ such that $\bd K$ has no fixed points, then the interior of $K$ has a unique invariant connected component, which has fixed point index $1$.
\end{corollary}

\subsection{The rotational attractors and repellors of a disk}\label{sec:A-R}
Given an invariant open topological disk $U$ without fixed points in its boundary and with a vanishing prime end rotation number (or equivalently, with nonpositive fixed point index), Theorem \ref{th:prime-end} allows us to define the families $\mc A(U)$ and $\mc R(U)$ of all rotational attractors and repellors (respectively) associated to fixed prime ends. Every element of $\mc A(U)$ (or $\mc R(U)$) is disjoint from $U$ and its boundary is the principal set of an attracting (repelling) prime end, so by Theorem \ref{th:matsu-naka} if the fixed point index in $U$ is $-k$, there are at most $k + 1$ elements in each set.
We remark that by Theorem \ref{th:prime-end}(4), the interior of each element of $\mc A(U)\cup \mc R(U)$ has exactly one invariant connected component, of index $1$.

\begin{lemma}\label{lem:A-R-disjoint} The elements of $\mc A(U) \cup \mc R(U)$ are pairwise disjoint cellular continua.
\end{lemma}
\begin{proof}
Let $K$ and $K'$ be two different elements of $\mc A(U) \cup \mc R(U)$. Since they are principal sets of different prime ends, by Theorem \ref{th:prime-end} there exist disjoint embedded rays $\Gamma$ and $\Gamma'$ such that $\Gamma$ spirals towards $\Fomega(\Gamma) = K$ and $\Gamma'$ spirals towards $\Fomega(\Gamma') = K'$ (and each is either a positive or a negative translation ray).
Assume for a contradiction that $K\cap K'\neq \emptyset$. Since $U$ is disjoint form $K$ and $K'$, it cannot be the case that $\Gamma\subset K'$ or $\Gamma'\subset K$, so only the last item of Lemma \ref{lem:spiral}(2) may hold. Thus $K\subset K'$, but interchanging $\Gamma$ and $\Gamma'$ we also conclude $K'\subset K$, a contradiction.
\end{proof}

\begin{lemma}\label{lem:A-R-same} Suppose that $U$ and $U'$ are disjoint invariant open topological disks without fixed points in their boundaries and with vanishing prime ends rotation number. If the basin of $A\in \mc A(U)$ intersects $U'$ then either $U'\subset A$, or $A\in \mc{A}(U')$. A similar statement holds replacing $\mc{A}$ by $\mc{R}$.
\end{lemma}
\begin{proof}
Let $\Gamma$ be a translation ray in $U$ spiraling towards $A$, which implies that $\Fomega(\Gamma)=A$. 
Assume $U'$ is not contained in $A$. Then $U'$ must be disjoint from $A$, since otherwise it would contain an arc joining a point of $\SS^2\sm A$ to a point of $A$, and therefore $U'$ would intersect $\Gamma$ (hence $U$) by Lemma \ref{lem:attractor-intersect}. 
Since $U'$ intersects the basin of $A$ and so does $U$, we may find a point $x\in U'$ arbitrarily close to $\bd U'$ in the basin of $A$. By Theorem \ref{th:matsu-naka}, such $x$ may be chosen in the basin of some attracting prime end $\pe p$ of $U'$, and therefore to some translation ray $\Gamma'$ in $U'$ converging towards $\pe p$ in $\cPE(U')$. By Theorem \ref{th:prime-end}, the set $A' = \Fomega(\Gamma')$ is an element of $\mc A(U')$ and $\Gamma'$ spirals towards $A'$. Note that since $x$ belongs to the basin of $A$ (but not to $A$), one has $\bd A\cap \bd A'\neq \emptyset$. Thus $\omega(\Gamma)$ intersects $\omega(\Gamma')$, and noting that $\Gamma'$ is disjoint from $\Fomega(\Gamma)\cup \Gamma$ we conclude from Lemma \ref{lem:spiral} that $\omega(\Gamma)\subset \omega(\Gamma')$ and $A = \Fomega(\Gamma)\subset \Fomega(\Gamma') = A'$.

We claim that $U$ cannot be contained in $A'$. Indeed, by Theorem \ref{th:prime-end}(4), the interior of any element of $\mc A(U)$ or $\mc A(U')$ has a unique invariant connected component. Since $A\cap U = \emptyset$, the unique invariant connected component $V$ of the interior of $A$ is disjoint from $U$. Moreover, one has $V\subset A'$ and $\bd V\subset \bd A'$, since $A\subset A'$ and $\bd A = \omega(\Gamma)\subset \omega(\Gamma')=\bd A'$.  These facts imply that $V$ is also a connected component of the interior of $A'$, and therefore it must be the unique invariant connected component of the interior of $A'$. But if $U\subset A'$, then $U$ is contained in some invariant connected component of the interior of $A'$, which means that $U\subset V$, contradicting the fact that $V$ is disjoint from $U$. Thus $U\not\subset A'$.

Hence, as in the beginning of the proof, Lemma \ref{lem:attractor-intersect} implies that $U$ is disjoint from $A'$, and we may repeat the previous argument interchanging $U$ with $U'$ and $A$ with $A'$ to conlcude that $A'\subset A$. Hence $A=A'\in \mc A(U')$.
\end{proof}

\subsection{Proof of Theorem \ref{th:disks}}

Suppose that $U\subset \SS^2$ is an $f$-invariant open topological disk such that $\rho(f, U) = 0$ but there are no fixed points in $\bd U$. As explained in Section \ref{sec:A-R}, there exist families $\mc{A}= \mc{A}(U)$ and $\mc{R} =  \mc{R}(U)$ of rotational attractors and repellors disjoint from $U$, each having at least one and at most $k+1$ elements (where $-k$ is the fixed point index of $U$), and elements of $\mc{A}\cup \mc{R}$ are pairwise disjoint. Moreover the boundary of any such element is the principal set of a fixed prime end of $U$.

Let $U_0, U_1,\dots, U_m$ be the connected components of $\SS^2\sm \bd U$ which have nonpositive index and contain a fixed point, where $U_0 = U$. Note that the nonpositive index implies that each $U_i$ has vanishing prime ends rotation number, so we also have families $\mc{A}(U_i), \mc{R}(U_i)$ of the corresponding rotational attractors and repellors of $U_i$. Since $\bd U_i\subset \bd U$, we see that $U$ intersects the basin of every element of $\mc A(U_i)$, so Lemma \ref{lem:A-R-same} implies that for any such element $A$, either $U\subset A$ or $A\in \mc A = \mc A(U)$. However $U\subset A$ is not possible, since it would imply that $U$ is an invariant connected component of the interior of $A$ with nonpositive index, contradicting Theorem \ref{th:prime-end}(4). Thus we have $\mc A(U_i)\subset \mc A$, and a similar argument shows $\mc R(U_i)\subset \mc R$.

Let $-k_i\leq 0$ denote the fixed point index of $U_i$. By Lemma \ref{lem:flower}, for each $i$ there exists an open topological disk $U_i'\subset U_i$ whose boundary in $\cPE(U_i)$ is as in Figure \ref{fig:flower}, \ie it consists of $2k_i+2$ fixed prime ends, $k_i+1$ repelling and $k_i+1$ attracting, together with $2k_i+2$ embedded translation lines, each connecting a repelling prime end to an attracting prime end, and such that $U'\sm U$ has no fixed points.
Let $\mc{F}$ denote the family of all such translation lines, for all $i\in \{0,\dots, m\}$.
Note that all elements of $\mc{F}$ are pairwise disjoint and disjoint from elements of $\mc{A}\cup \mc{R}$.
Moreover, by Theorem \ref{th:prime-end} every line in $\mc{F}$ spirals from an element of $\mc{R}$ to an element of $\mc{A}$.

Since rotational attractors and repellors are cellular continua, we may collapse each element of $\mc{A}\cup \mc{R}$ to a point; \ie letting $K$ be the union of all elements of $\mc{A}\cup \mc{R}$, we may regard the set $\SS^2\sm K$ as $\til{S}\sm \til K$, where $\til S$ is a sphere and $\til{K}$ is a finite set of punctures, each induced from a corresponding element of $\mc{A}\cup \mc{R}$ ($\til S$ is the Freudenthal compactification of $\SS^2\sm K$). The map $f$ induces a homeomorphism $\til f$ of $\til S$, which coincides with $f$ on $\til S\sm \til K = \SS^2\sm K$ and fixes elements of $\til K$ pointwise. Moreover, letting $\til K_A$ and $\til K_R$ denote the elements of $\til K$ obtained from collapsing elements of $\til A$ or $\til R$, accordingly, each element of $\til K_A$ is an attracting fixed point and each element of $\til K_R$ is a repelling fixed point for $\til f$. Each line $\Gamma\in \mc{F}$, seen as a subset of $\til S$, connects a repelling fixed point from $\til K_R$ to an attracting fixed point from $\til K_A$. Consider the set $\Sigma = \bigcup_{\Gamma\in \mc F} \Gamma$, and note that $G = \til K\cup \Sigma$ is a bipartite planar graph in $\til S$ with vertices in $\til K$ and edges in $\mc F$, each joining a vertex in $\til K_R$ to a vertex in $\til K_A$. See Figure \ref{fig:collapse}. Note that $\til S\sm G = \SS^2\sm (\Sigma \cup K)$, so each connected component of $\til S\sm G$ is a connected component of $\SS^2\sm (\Sigma \cup K)$.

\begin{figure}[ht!]
\includegraphics[height=4cm]{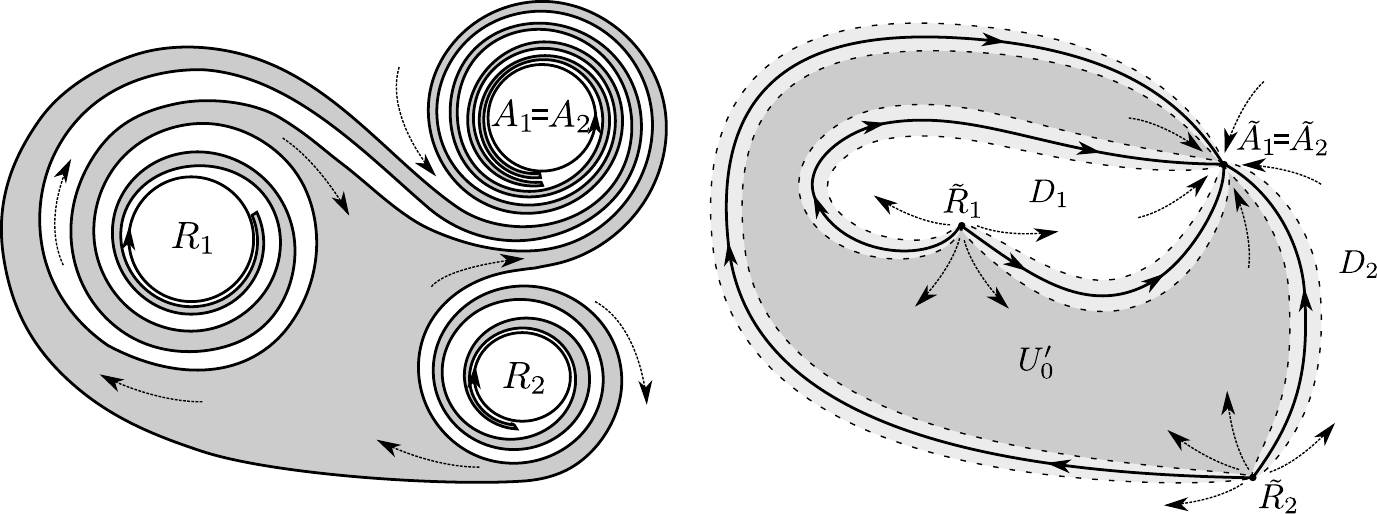}
\caption{An example before and after collapsing. The dotted lines are the elements of $\mc F$.}
\label{fig:collapse}
\end{figure}

Note that every connected component of $\til S \sm G$ is invariant, and moreover Lemma \ref{lem:index-LC} implies that any such component has nonpositive index. 

We claim that the only connected components of $\til S\sm G$ which contain a fixed point are the sets $U_i'$, for $i\in\{0,\dots, m\}$. Indeed, if $D$ is such a component containing a fixed point, regarding $D$ as a subset of $\SS^2$, we have that $\fix(\til f)\cap D$ is compact (otherwise there would exist a sequence of fixed points in $\SS^2\sm K$ converging in $\SS^2$ to a point of $\bd K$, contradicting the fact that $\bd K$ has no fixed points). Hence $\fix(f)\cap D$ must be contained in a finite union of invariant connected components of $\SS^2\sm \bd U$, each containing at least one fixed point. Since the fixed point index in $D$ is nonpositive, some such component must have nonpositive index, which means that it is equal to $U_j$ for some $j$. Thus $U_i\cap D$ contains some fixed point $p$ of $f$. Moreover, $U_i'\subset U_i$ is a connected component of $\til S\sm G$, and $p\in U_i'\cap D$ (since $U_i\sm U_i'$ has no fixed points).  Since $D$ is also a connected component of $\til S\sm G$ and intersects $U_i'$, we conclude that $D = U_i'$, as we wanted.

For example, in Figure \ref{fig:collapse}, the components of the complement of $\Sigma \cup K$ are $U_0'$, together with four index $0$ disks intersecting the boundary of $U$ and two additional disks $D_1, D_2$ disjoint from $U$ (also of index $0$ in this case). As we just showed, the only components that could have fixed points are $U_0'$ and the components disjoint from $U$ (\ie $D_1$ and $D_2$). 

As we have shown, the only connected components of $\til{S}\sm G$ having fixed points are $U_0',\dots U_m'$. Note that by definition none of the sets $U_i'$ intersects $\bd U$. Thus every connected component $D$ of $\til{S}\sm G = \SS^2\sm (K\cup \Sigma)$ intersecting $\bd U\sm K$ is fixed point free, and therefore has index $0$. By Lemma \ref{lem:index-LC}, the boundary of $D$ in $\til S$ consists of exactly two elements of $\mc F$ and two elements $\til K$ (one in $\til K_A$ and one in $\til K_R$).  Denoting by $T$ the closure of $D$ in $\til S$ with the two fixed points removed, one deduces from Lemma \ref{lem:nofix-trans} that $T$ is an embedded translation strip for $\til f$. But since $T$ is also a subset of $\SS^2$, we see that $T$ is an embedded translation strip for $f$ as well. Moreover, the boundary lines of $T$ both spiral from some $R\in \mc R$ to some $A\in \mc A$, which implies that $T$ has the same property. From these facts one concludes that the filled $\omega$-limit of $T$ is $A$ and its filled $\alpha$-limit is $R$. 

Thus, denoting by $\mc{T}$ the family of all such translation strips (so that the interior of every $T\in \mc T$ is a connected component of $\til S\sm G$ intersecting $\bd U\sm K$), we have that the interiors of elements of $\mc T$ cover $\bd U\sm K$, and in $\SS^2$ each $T\in \mc T$ spirals from an element of $\mc R$ to an element of $\mc A$. 
We claim that one of the boundary lines of $T$ is contained $U$. To see this, note that the interior of $T$ intersects $U$ (because it intersects $\bd U \sm K$), and $U$ cannot be entirely contained in $\inter T$ (since $\inter T$ is disjoint from the elements of $\mc F$, while $U=U_0$ contains $2+2k_0\geq 2$ elements of $\mc F$). Thus $U$ intersects $\bd T$, which consists of the two boundary lines and a subset of $K$. Since $U$ is disjoint from $K$, it follows that $U$ intersects some boundary line of $T$, which implies that this line is contained in $U$ (since every element of $\mc F$ is contained in some $U_i$).
We further claim that an element of $\mc F$ contained in $U$ cannot be the boundary line of two different strips $T, T'\in \mc T$. Indeed, if this is the case then $T\cup T'$ is a new strip, and again its interior cannot contain $U$ entirely since it only contains one element of $\mc F$, whereas $U$ contains at least $2+2k_0 \geq 2$. 
Thus the number of elements of $\mc T$ is at most the number of elements of $\mc F$ contained in $U$, which is $2+2k_0$. Some elements of $\mc T$ could be bounded by two elements of $\mc F$ contained in $U$ (for instance if it contains the disk $D_1$ in Figure \ref{fig:collapse}), but we may also conclude that $\mc T$ has at least $1+k_0$ elements.

Finally, to see that every element $\mc{A}\cup \mc R$ is accumulated by some strip $T\in \mc T$, let $A\in \mc A$ and  note that there exists $x\in \bd U \sm K$ arbitrarily close to $A$, in particular in the basin of $A$. If $T\in \mc T$ is the element containing $x$, since $T$ must be entirely contained in the basin of a unique element of $\mc{A}$ (which is the filled $\omega$-limit of $T$), it follows that $A$ is this element and $T$ spirals towards $A$. A similar argument applies to elements of $\mc R$.

This completes the proof of Theorem \ref{th:disks}. \qed

\begin{remark}\label{rem:separa} One has the additional property that for every $T\in \mc T$, the boundary lines of $T$ are separated by $\bd U\cap T$ in $T$. Indeed, if this is not the case then the two boundary lines belong to the same connected component of $\SS^2\sm \bd U$, which must be $U=U_0$ since we already know that one of the two lines lies in $U$. In the prime ends compactification of $U$, these two lines connect the same two fixed prime ends, so by the description from Lemma \ref{lem:model} the interior of $T$ must be $U_0'$, which is not possible since the interior of $T$ intersects $\bd U$. 
\end{remark}

\subsection{A monotone semiconjugation to a planar graph}\label{sec:graph}
Recall that by a \emph{planar graph} $G$ in $\SS^2$ we mean a finite set of vertices $V(G)$ (points) and edges $E(G)$ (lines) each connecting two vertices, such that the edges are pairwise disjoint and $G=V\cup E$.

A planar graph $G\subset \SS^2$ is an attractor-repellor graph for a homeomorphism $F\colon \SS^2\to \SS^2$ if every vertex of $G$ is an attracting or repelling fixed point of $F$, every edge is a translation line joining a repellor to an attractor, and in addition every vertex has even degree. 
Note that in particular a neighborhood of $G$ is contained in the union of the basins of the attracting and repelling fixed points.

Let us briefly explain how one may use the reduction from the previous proof to obtain a monotone map $h\colon \SS^2\to \SS^2$ which maps $\bd U$ to an attractor-repellor graph $G$ of a homeomorphism $F$ such that $hf = Fh$.  

To see this, let us first define a map $h_0\colon \SS^2\to \til{S}$ which collapses the elements of $\mc{A}\cup \mc{R}$ as done in the previous section. We identify $\til{S}$ with $\SS^2$ (by composing $h_0$ with a homeomorphism). If $\til K$ denotes the union of the images by $h_0$ of the attractors and repellors (which are points), as we have seen, $f$ induces a map $f_0$ such that each element of $\til K$ is an attracting or repelling fixed point. Moreover, $h_0(U)$ is an invariant disk and the boundary of $h_0(U)$ is covered by the interiors of a finite family of translation strips $\mc T$, each connecting a repelling element of $\til K$ to an attracting one. For each $T\in \mc T$, there is a natural way of defining a monotone map that collapses it into a line: one may find an invariant foliation of $T$ by ``transverse'' arcs (by foliating a fundamental domain first and then extending in the obvious way), and since any such arc converges to a fixed point when iterated by $f_0$ (both in the future and in the past), one may easily see that the map which collapses these arcs to points maps $T$ to a line connecting two points, which is a translation line for the map induced by $f_0$. Since the elements of $\mc T$ are pairwise disjoint, we may do this for each $T\in \mc T$ obtaining a monotone map $h_1\colon \SS^2\to \SS^2$ which maps each translation strip to a translation line for an induced map $F\colon \SS^2\to \SS^2$, satisfying $h_1f_0 = Fh_1$. Finally, letting $h = h_1h_0$ we see that $h$ and $F$ have the required properties, where $E = \{h(T): T\in \mc T\}$, $V = \{h(C) : C\in \mc{A}\cup \mc{R}\}$ and $G = E\cup V$.

Thus we have the following
\begin{theorem}\label{th:graph} Under the hypotheses of Theorem \ref{th:disks}, there exists a monotone surjection $h\colon \SS^2\to \SS^2$, and an orientation-preserving homeomorphism $F\colon \SS^2\to \SS^2$ such that:
\begin{itemize}
\item $hf = Fh$;
\item $h$ is injective outside a neighborhood of $\bd U$;
\item $G = h(\bd U)$ is an attractor-repellor graph for $F$.
\end{itemize}
\end{theorem}

Recall that each element of $T$ spirals from some $R_T\in \mc R$ to some $A_T\in \mc A$. Let us define
$$K_T = (T\cap \bd U) \cup \bd A_T \cup \bd R_T.$$
We claim that $K_T$ is a continuum. Indeed, $T$ is clearly compact, since $T$ accumulates only on $\bd R_T$ and $\bd A_T$. If $K_T$ is not connected, then we may write $K_T = C_1 \cup C_2$ for two nonempty compact proper subsets $C_1, C_2$ of $K_T$ such that $C_1\cap C_2=\emptyset$. But since $T\cap \bd U$ separates the two boundary components of $T$ (see Remark \ref{rem:separa}, some connected component $C$ of $T\cap \bd U$ must also separate the two boundary components (see \cite[Theorem 14.3]{newman}). The connected set $C$ must be contained in $C_1$ or $C_2$. Assume without loss of generality $C\subset C_1$. Then the closure of $C_1$ intersects both $\bd A_T$ and $\bd R_T$, and since the latter sets are connected, $\bd A_T\cup \bd R_T\subset C_1$. Hence $C_2$ is disjoint from $\bd A_T\cup \bd R_T$ (and therefore from $A_T\cup R_T$). But then $C' = C_1\cup \bd U \sm \inter T$ is a compact nonempty proper subset of $\bd U$ such that $C' \cup C_2 = \bd U$, contradicting the fact that $\bd U$ is connected. 

We define a \emph{basic block} for $f$ any continuum $K$ which consists of the disjoint union of the boundaries of a rotational repellor $R(K)$, a rotational attractor $A(K)$, such that $K\sm (R(K)\cup A(K))$ is contained in some translation line spiraling from $R(K)$ to $A(K)$ (see Figure \ref{fig:block}). We call $A(K)$ and $R(K)$ the attracting and repelling nodes of $K$.
\begin{figure}[ht!]
\includegraphics[height=4cm]{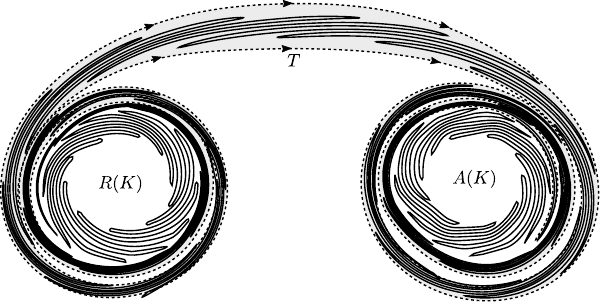} 
\caption{A basic block.}
\label{fig:block}
\end{figure}

Note that, by construction of $h$, if $\Gamma$ is an edge of $G$ then $h^{-1}(\ol{\Gamma})\cap \bd U$ is a basic block. Hence we have:

\begin{theorem}\label{th:blocks} Under the hypotheses of Theorem \ref{th:disks}, $\bd U$ is the union of at least $k+1$ and at most $2k+2$ basic blocks, where $-k = i(f, U)$, such that any two basic blocks intersect at most at their attracting or repelling nodes (in which case the corresponding nodes coincide). Moreover, the map $h$ from Theorem \ref{th:disks} maps each basic block to a different edge of the graph $G$ (including endpoints).
\end{theorem}

An application is the following result about indecomposable boundaries:
\begin{corollary}\label{coro:decomposable-index} If $\bd U$ is an indecomposable continuum without fixed points, then the fixed point index in $U$ is either $0$ or $1$.
\end{corollary}
\begin{proof} If the index is not $1$, then by Corollary \ref{coro:disk-index} the prime ends rotation number in $\bd U$ is $0$ and the fixed point index in $U$ is nonpositive. By the previous theorem, if $i(f,U)<0$ then there are at least $2$ basic blocks in $\bd U$, from which the decomposability follows easily.
\end{proof}

Note that in the previous corollary, if $U$ has index $0$ the only possibility is that $\bd U$ consists of a single basic block, as in Figure \ref{fig:block}.

\subsection{Circloids: proof of Theorem \ref{th:circloid} and Corollary \ref{coro:circloid-decomposable}} \label{proof-circloids}

Recall that a continuum $C\subset \A$ is \emph{essential} if its complement in $\A$ has two unbounded connected components, and \emph{annular} if these are the only connected components of its complement. Moreover, an annular continuum $C$ is a \emph{circloid} if no proper essential subcontinuum is annular.
It is easy to verify that the latter condition is equivalent to saying that $\bd C$ coincides with the boundary of each of the two (unbounded) components of $\A\sm C$.

If $C\subset \A$ is an essential circloid invariant by a homeomorphism $h\colon \A\to \A$, we may compactify $\A$ with two points to obtain a homeomorphism $f\colon \SS^2 \to \SS^2$ which leaves $C$ invariant and has two invariant disks $U_-$ and $U_+$ such that $\bd U_- = \bd U_+ = \bd C$.

Assume as in the statement of Theorem \ref{th:circloid} that $C$ contains a fixed point but $\bd C$ does not. Then some connected component of $\SS^2\sm \bd C$ other than $U_-$ and $U_+$ is invariant and contains a fixed point. The number of components of $\SS^2\sm C$ containing fixed points must be finite, otherwise there would exist a sequence of fixed points accumulating in $\bd C$. 
Moreover, by Corollary \ref{coro:disk-index} the index of any such component is at most $1$. Since the sum of their indices is $2$ and there are more than two of them, at least one of these components has nonpositive index.
Let $\mc U$ be the set of all connected components of $\SS^2\sm \bd C$ which have nonpositive index and contain a fixed point of $f$. For any $U\in \mc U$, there are no fixed points in $\bd U$ and since the fixed point index is nonpositive, Corollary \ref{coro:disk-index} implies that $\rho(f,U)=0$. Thus, as in Section \ref{sec:A-R} we may define the sets $\mc{A}(U)$ and $\mc{R}(U)$ for each $U\in \mc U$. 
Let $\mc{A} = \bigcup_{U\in \mc{U}} \mc{A}(U)$ and $\mc{R} = \bigcup_{U\in \mc{U}} \mc{R}(U)$. 
From Theorem \ref{th:prime-end} we know that the interior of each element of $\mc{A}\cup \mc{R}$ has a unique invariant connected component, of index $1$. This implies that elements of $\mc{U}$ are disjoint from elements of $\mc{A}\cup \mc{R}$, since elements of the latter set are bounded by $\bd C$, so any $U\in \mc U$ intersecting some $A\in \mc{A}$ would have to be entirely contained in $A$, and moreover it would be a connected component of the interior of $A$ (and since $U$ has nonpositive index, this is not possible). 

In addition, if $U, U'$ are two different elements of $\mc U$, then given any $A\in \mc A(U)$ and $A'\in \mc A(U')$ one has that either $A=A'$ or $A\cap A'=\emptyset$. This follows from Lemma \ref{lem:A-R-same}, since as we just saw no element of $\mc{U}$ can be contained in an element of $\mc{A}\cup \mc{R}$. Thus $\mc{A} \cup \mc{R}$ is a pairwise disjoint family and its elements are  disjoint from all elements of $\mc{U}$. 

Moreover, $U_-$ and $U_+$ are also disjoint from elements of $\mc A \cup \mc R$. Indeed, as in the previous paragraph if for instance $U_-$ intersects some $A\in \mc A$, then $U_-\subset A$; but since $\bd U_- = \bd C$ it follows that $\bd C\subset A$, and so the only connected component of $\SS^2\sm \bd C$ disjoint from $A$ has to be $\SS^2\sm A$. But since there exists an element $U\in \mc U$ such that $A\in \mc A(U)$, and for such $U$ there exists at least one element $R\in \mc{R}(U)$ which is disjoint from $A$ (by Lemma \ref{lem:A-R-disjoint}), we arrive to a contradiction.

As in the proof of Theorem \ref{th:disks}, we may choose for each $U\in \mc{U}$ an open topological disk which we denote by $U^*$, such that $U\sm U^*$ has no fixed points and $U^*$ is bounded in $U$ by translation lines, each spiraling from an element of $\mc{R}(U)$ to an element of $\mc{A}(U)$ (as in Lemma \ref{lem:model}); the number of such lines is $2-2k$, where $k$ is the index of $U$. Let $\mc{U}^*$ be the family of all such disks $U^*$, and $\mc{F}$ the family of all boundary lines of such disks. Considering the sphere $\til S$ obtained by collapsing elements of $\mc A$ and $\mc R$ to points, we have that the elements of $\mc{F}$ are the edges of a planar graph $G$ in $\til S$ with vertices in the set $\til K = \til K_A\cup \til K_R$, where $\til K_A$ and $\til K_R$ are the points obtained from collapsing elements of $\mc A$ and $\mc R$, which are attracting and repelling fixed points for the induced dynamics $\til f$, respectively. The graph $G$ is a (not necessarily connected) attractor-repellor graph for $\til f$. Note that all elements of $\mc U^*$ are faces of $G$, and each element of $\mc U$ contains one (and only one) such face. 

Let $\til{C}\subset \til S$ be the set obtained from $C$ after collapsing the elements of $\mc A\cup \mc R$. Note that $C \sm \bigcup_{A\in \mc A\cup \mc R} A = \til C \sm \til K$, and the connected components of $\SS^2\sm C$ not contained in elements of $\mc A\cup \mc R$ are the connected components of $\til S\sm \til C$, so elements of $\mc U$ can be regarded as open topological disks both in $\SS^2$ and in $\til S$. Note that the elements of $\mc U^*$ are faces of $G$ (each contained in some element of $\mc U$).

\begin{claim}\label{claim:circloid-collapse} $\bd_{\til S} \til C$ is the common boundary of $U_-$ and $U_+$ in $\til S$, \ie $$\bd_{\til S} U_- = \bd_{\til S} C = \bd_{\til S} U_+.$$
\end{claim}
\begin{proof}
Letting $h\colon \SS^2\to \til S$ be the (continuous) map collapsing elements of $\mc A\cup \mc R$, the preimage by $h$ of a point $x$ in the boundary of $\til C$ always contains a point $x'$ in the boundary of $C$ (since either $h^{-1}(x)$ contains an element of $\mc{A}\cup \mc{R}$, which intersects $\bd C$, or it does not, in which case $h^{-1}(x)$ is a single point and $h|_{h^{-1}(x)}$ is a local homeomorphism).
Since $x'$ must belong to the boundaries of $U_-$ and $U_+$ in $\SS^2$, and $h$ is injective on $U_-$ and $U_+$, it follows that $x = h(x')$ belongs to the boundary of $U_-$ and $U_+$ in $\til S$ as well, proving our claim that $\bd_{\til S} \til C$ is the common boundary of $U_-$ and $U_+$ in $\til S$.
\end{proof}

\begin{claim}\label{claim:face-contained} Every connected component of $\til S\sm \bd_{\til S} \til C$ that is not an element of $\mc{U}$ is entirely contained in some face of $G$. 
\end{claim}
\begin{proof} If $U$ is any such component, since $U\notin \mc{U}$ it does not intersect any edge of $G$. Since $U$ does not contain any vertex of $G$ either (as vertices belong to $\til C$) the claim follows.
\end{proof}

\begin{claim} Both sets $U_-$ and $U_+$ belong to $\mc{U}$
\end{claim}
\begin{proof}
Suppose for instance that $U_-\notin \mc U$. Then by the previous claim we have that $U_-$ is contained in some face $D_-$ of $G$. The fixed point set in $D_-$ is nonempty since it contains $-\infty$, and it may be covered by finitely many connected components of $V_1,\dots, V_m$ of $\til{S}\sm \bd_{\til S}\til C$. None of the sets $V_i$ may belong to $\mc U$, since otherwise $V_i^*$ would be a face of $G$ intersecting $D_-$ (hence equal to $D_-$) which would imply that $U_-\subset V_i^*\subsetneq V_i$, a contradiction since both $V_i$ and $U_-$ are connected components of $\til{S}\sm \bd_{\til S}\til C$. 

Since each $V_i$ intersects $D_-$ and is not in $\mc U$, by the previous claim $V_i\subset D_-$ for $i\in \{1,\dots, m\}$. Note that $V_i$ may be regarded as a connected component of $\SS^2\sm \bd C$ as well, and since $V_i\notin \mc U$ and $V_i$ has a fixed point, the definition of $\mc U$ implies that the fixed point index of $V_i$ is positive. Thus the fixed point index of $D_-$ is at least $m$, in particular positive, contradicting Lemma \ref{lem:graph-index}.
\end{proof}

We have thus shown that $U_-$ and $U_+$ belong to $\mc U$, which means that they have nonpositive index. 
By Theorem \ref{th:disks} applied  to $U_-$ there is a finite family $\mc T$ of pairwise disjoint translation strips which, seen in $\til S$, each joins a point of $\til K_R$ to a point of $\til K_A$, and such that the interiors of elements of $\mc T$ cover $\bd_{\til S} U_-\sm \til K$. Each $T\in \mc T$ intersects $\bd_{\til S} U_- = \bd_{\til S} \til C = \bd_{\til S} U_+$, so it also intersects $U_+$. On the other hand the two boundary lines of $T$, together with two points of $\til K$, bound a loop, and one of the boundary lines of $T$ lies in $U_-$ which is disjoint from $U_+$. Thus $U_+$ must intersect the remaining boundary line of $T$, and since this line is disjoint from $\bd_{\til S} \til C$ it must be entirely contained in $U_+$. In other words, each element of $T$ has one boundary line in $U_-$ and the other in $U_+$. Moreover, the boundary lines in $U_-$ bound a region $D_-$ in $U_-$ similar to the set $U_-^*$ used before (\ie as in Lemma \ref{lem:model}); this is clear from the proof of Theorem \ref{th:disks}. In particular, there are $2k+2$ such lines in $U_-$, where $-k$ is the fixed point index in $U_-$, and $\mc T$ has $2k+2$ elements.

If $G'$ denotes the graph whose edges are boundary lines of elements of $\mc{T}$ (whose vertices are in $\til K$), then $D_-$ is a face of $G'$ contained in $U_-$. Moreover, there is a face $D_+$ of $G'$ contained in $U_+$ whose boundary edges are precisely the boundary lines of elements of $\mc T$ in $U_+$. For instance to see this note that for each $T\in \mc T$, the set $T\cap U_+$ is a cross-section of $U_+$, together with the cross-cut defining it, and any two such cross-sections are disjoint. The complement of all these cross-sections in $U_+$ is a topological disk $D_+$ which is also a face of $G'$ contained in $U_+$ (note that from these remarks one may deduce that the fixed point index of $U_+$ is also $-k$).

Hence $\til S\sm (D_-\cup D_+)$ is the union of all elements of $\mc T$ and $\til K$, which means that $\SS^2\sm (D_-\cup D_+)$ is the (disjoint) union of the elements of $\mc{A}$, $\mc{R}$ and $\mc{T}$ (note that this set contains $C$).
Moreover, the boundaries of $D_-$ and $D_+$ have $2k+2$ edges each, and $G'$ has $2k+4$ faces, so using Euler's formula we see that there are $2k+2$ vertices. Since the boundary of $D_-$ is a bipartite graph with $2k+2$ edges and each vertex of $G'$ belongs to some such edge, it follows that $\til{K}_A$ and $\til{K}_R$ have $k+1$ elements each and $\bd D_-$ is a simple loop. Thus there are $k+1$ elements in each set $\mc{A}$ and $\mc{R}$.

This proves the first part of Theorem \ref{th:circloid}. In fact, these arguments imply that the elements of $\mc T$ are cyclically ordered; more precisely, each element of $\mc T$ joins an element of $\mc{A}$ to an element of $\mc{R}$ and every element of $\mc{A}\cup\mc{R}$ is accumulated by exactly two elements of $\mc T$. From this observation and Theorem \ref{th:graph} we obtain the second part of Theorem \ref{th:circloid}; namely, the attractor-repellor graph from Theorem \ref{th:graph} must map $C$ to a cyclic graph, \ie a circle with Morse-Smale dynamics. 

Moreover, from Theorem \ref{th:blocks} we see that $C$ must is decomopsable, since there are at least two basic blocks, so Corollary \ref{coro:circloid-decomposable} is proved as well.
\qed

\def\MR#1{} % do not show mrnumber in bibliography
\bibliographystyle{koro} 
\bibliography{circloid2}

\end{document}